\documentclass{article}
\usepackage[hidelinks]{hyperref}
\usepackage[a4paper,top=3cm,bottom=3.5cm,left=3cm,right=3cm,marginparwidth=1.75cm]{geometry}
\usepackage{amsmath, amssymb, amsthm, graphicx, cleveref, tikz, caption, subcaption, epstopdf}
\graphicspath{{images/}}
\usetikzlibrary{calc,angles,quotes,arrows.meta,patterns}

\usepackage[utf8]{inputenc}

\usepackage{todonotes,comment}

\newcommand{\ddp}[2]{\frac{\partial#1}{\partial#2}}

\newcommand{\D}{\partial D}
\newcommand{\B}{\mathcal{B}}
\renewcommand{\S}{\mathcal{S}}
\newcommand{\K}{\mathcal{K}}

\renewcommand*{\Re}{\operatorname{Re}}
\renewcommand*{\Im}{\operatorname{Im}}
\newcommand{\A}{\mathcal{A}}
\newcommand{\de}{\: \mathrm{d}}
\newcommand{\R}{\mathbb{R}}
\newcommand{\iu}{\mathrm{i}\mkern1mu}
\newcommand{\C}{\mathcal{C}}
\graphicspath{{images/}}
\newcommand{\vb}{\mathbf{v}}
\newcommand{\w}{\mathbf{w}}
\renewcommand{\k}{\mathbf{k}}
\renewcommand{\P}{\mathcal{P}}
\newcommand{\I}{\mathcal{I}}

\renewcommand{\c}{\mathbf{c}}

\newcommand{\Z}{\mathbb{Z}}
\newcommand{\p}{\partial}
\renewcommand{\epsilon}{\varepsilon}
\newcommand{\dx}{\: \mathrm{d}}

\newcommand{\ie}{\textit{i.e.}}
\newcommand{\nm}{\noalign{\smallskip}}
\newcommand{\ds}{\displaystyle}

\newtheorem{thm}{Theorem}
\newtheorem{prop}[thm]{Proposition}
\newtheorem{lemma}[thm]{Lemma}

\newtheorem{remark}[thm]{Remark}

\numberwithin{equation}{section}
\numberwithin{thm}{section}

\title{Bound states in the continuum and {Fano} resonances in subwavelength resonator arrays}

\author{Habib Ammari\thanks{\footnotesize Department of Mathematics, ETH Z\"urich, R\"amistrasse 101, CH-8092 Z\"urich, Switzerland (habib.ammari@math.ethz.ch, bryn.davies@sam.math.ethz.ch, erik.orvehed.hiltunen@sam.math.ethz.ch).}\and Bryn Davies\footnotemark[1]  \and Erik Orvehed Hiltunen\footnotemark[1] \and  Hyundae Lee\thanks{\footnotesize  Department of Mathematics, Inha University,  253 Yonghyun-dong Nam-gu,  Incheon 402-751,  Korea (hdlee@inha.ac.kr).} \and Sanghyeon Yu\thanks{\footnotesize Department of Mathematics, Korea University, Seoul 02841, S. Korea (sanghyeon\_yu@korea.ac.kr).}}

\date{}

\begin{document}
	\maketitle
	
\begin{abstract}
	When wave scattering systems are subject to certain symmetries, resonant states may decouple from the far-field continuum; they remain localized to the structure and cannot be excited by incident waves from the far field. In this work, we use layer-potential techniques to prove the existence of such states, known as bound states in the continuum, in systems of subwavelength resonators. When the symmetry is slightly broken, this resonant state can be excited from the far field. Remarkably, this may create asymmetric (Fano-type) scattering behaviour where the transmission is fundamentally different for frequencies on either side of the resonant frequency. Using asymptotic analysis, we compute the scattering matrix of the system explicitly, thereby characterizing this Fano-type transmission anomaly.
\end{abstract}
	\vspace{0.5cm}
	\noindent{\textbf{Mathematics Subject Classification (MSC2000):} 35J05, 35C20, 35P20, 35P30.
		
		\vspace{0.2cm}
		
		\noindent{\textbf{Keywords:}} bound state in the continuum, Fano resonance, subwavelength resonance, metamaterials, scattering, metascreen, capacitance matrix
		\vspace{0.5cm}

	\section{Introduction}
	
	The existence of asymmetric peaks in transmission spectra is a curious phenomenon that has been studied at length in a variety of settings. Resonance peaks with this characteristic asymmetric shape are often known as \emph{Fano resonances} due to the work of Ugo Fano \cite{fano1961effects}, who observed this behaviour in the scattering of electrons by helium. Fano famously explained the asymmetric line shape as being due to the interference between a ``discrete state'' and a ``continuum''.
	
	The Fano-type resonance studied in this paper emerges from the interference between the two coupled resonant frequencies of a pair of resonators. In particular, we study resonator pairs which have been repeated periodically to form a metascreen. The first resonant frequency of this structure corresponds to the universal property that incoming plane waves with very low frequencies will be unaffected by the metascreen. This response occurs for a relatively broad range of frequencies and corresponds to the ``continuum'' of states that Fano described. In contrast, the second resonant frequency originates from the resonant behaviour of the metascreen. This peak depends heavily on the configuration of resonators within the metascreen and has a comparatively sharp response, corresponding to Fano's ``discrete state''. Since the width of the resonant peaks are proportional to the imaginary parts of corresponding resonances, the Fano-type resonance is characterised by the interference of two resonances with significantly different imaginary parts. By manipulating the parameters of the system, we are able to create interactions between these two resonant states, which leads to the creation of a Fano-type asymmetric transmission anomaly. Such an anomaly is depicted in \Cref{fig:lines}, where we sketch examples of transmission spectra (\textit{i.e}. the intensity of the transmitted field as a function of the frequency of the incident field). The difference between an asymmetric Fano-type line shape and a symmetric (\emph{Lorentzian}) profile is clear.
	
	The resonant properties of the metascreen can be fine-tuned by altering the parameters of the system. In the case that the metascreen is symmetric, we will prove that the second (sharper) resonant frequency is real.  Further, we will see that it corresponds to an eigenvalue that is embedded within the \emph{continuous radiation spectrum}, which is the spectrum of waves that can propagate into the far field.  Remarkably, we will show that the eigenmode associated with this real-valued resonant frequency vanishes in the far field, meaning that it will not interact with incoming waves and the corresponding resonance peak will therefore not appear in the transmission spectrum. A resonant state of this nature is known as a \emph{bound state in the continuum} and has a range of important applications in the design of lasers, filters and sensors \cite{hsu2016bound, kodigala2017lasing}.
	
	If we consider a metascreen that is no longer symmetric, then the real eigenvalue will be shifted into the lower complex plane and will correspond to a sharp peak in the transmission spectrum. The phase of the transmitted wave is different on each side of the peak and will interfere either constructively or destructively with the broad peak originating from the first resonant frequency. Thus, we obtain an asymmetric transmission spectrum that is characteristic of a Fano-type anomaly. This unusual transmission spectrum is not only of academic interest but has various applications, for example in the broadband manipulation of light \cite{wu2011broadband} and in the design of tunable sensors \cite{lassiter2010fano}.
	
	Phenomena similar to those studied in this work have have been studied in several other settings. For example, Fano-type anomalies have been observed in metallic gratings with repeated pairs of narrow slits \cite{lin2020mathematical, lin2020fano} as well as other dimerized structures \cite{limonov2017fano, song2019observation, zangeneh2019topological}. Likewise, bound states in the continuum have been both predicted theoretically and observed experimentally in periodic structures in photonics, optics, electrical circuits and quantum mechanics \cite{sadrieva2019experimental, kodigala2017lasing, hsu2016bound}. A variety of methods have been used to understand these phenomena including coupled-mode theory \cite{fan2003temporal}, analytic perturbation theory \cite{shipman2013resonant} and asymptotic methods \cite{lin2020mathematical, lin2020fano}.	
	
	In this work, we will study the existence of Fano-type resonances in systems of subwavelength resonators. That is, we will study a Helmholtz scattering problem posed on a system of material inclusions whose material parameters contrast greatly with those of the background medium. The main contribution of our work is a unified, mathematically rigorous, theory for both Fano-type resonances and bound states in the continuum. This extends the mathematical foundation of these phenomena, most notably the works \cite{lin2020mathematical, lin2020fano}, to the setting of high-contrast metamaterial crystals. We will perform asymptotic analysis in terms of the material contrast and define subwavelength resonant modes to be those whose frequencies converge continuously to zero in this limit \cite{ammari2021functional}. We will, first, recall asymptotic expressions for the subwavelength band structure in terms of the \emph{quasiperiodic capacitance matrix} (\Cref{thm:res_quasi}), before computing explicit expressions for the subwavelength band structure close to the origin, corresponding to the two resonances mentioned above (\Cref{thm:res0}). With this analysis in hand, we will prove that if the metascreen is symmetric then the second of these resonances is real (\Cref{prop:realeigen}) and the corresponding mode is a bound state in the continuum in the sense that it does not propagate into the far field (\Cref{prop:rad0}) and cannot be excited by waves incoming from the far field (\Cref{prop:trans0}). Finally, we will derive an expression for the scattering matrix which can be used to demonstrate the occurrence of a Fano-type transmission anomaly (\Cref{thm:phano}). This theoretical analysis is complemented by numerical simulations, which demonstrate Fano-type transmission anomalies for asymmetric structures (Figures~\ref{fig:fano}~to~\ref{fig:fanos3}) and a bound state in the continuum in the symmetric case (Figure~\ref{fig:fanoBIC}).

\begin{figure}
\begin{subfigure}[b]{0.45\linewidth}
	\centering
\begin{tikzpicture}[scale=1.5]
	\draw[line width=0.5mm,red] plot [smooth] coordinates {(0,0.08) (1.2,0.2) (1.5,1) (1.8,0.2) (3,0.08)};
	\draw[->,line width=0.4mm] (0,0) -- (3.2,0) node[pos=0.95, yshift=-8pt]{$\omega$};
	\draw[->,line width=0.4mm] (0,0) -- (0,1.2) node[pos=0.85, xshift=-8pt]{$T$};
\end{tikzpicture}
\caption{Lorentzian line shape}
\end{subfigure}
\begin{subfigure}[b]{0.45\linewidth}
	\centering
	\begin{tikzpicture}[scale=1.5]
		\draw[line width=0.5mm,red] plot [smooth] coordinates {(0,0.08) (1.2,0.2) (1.5,1) (1.7,0.04) (2,0.15) (3,0.1)};
		\draw[->,line width=0.4mm] (0,0) -- (3.2,0) node[pos=0.95, yshift=-8pt]{$\omega$};
		\draw[->,line width=0.4mm] (0,0) -- (0,1.2) node[pos=0.85, xshift=-8pt]{$T$};
	\end{tikzpicture}
	\caption{Fano-type line shape}
\end{subfigure}
\caption{When the frequency of incident waves varies, the transmittance $T$ (\textit{i.e.} the intensity of the transmitted field) will have peaks at certain frequencies. A Lorentzian line shape, which is a symmetric peak, is typically found in scattering problems. In this work, we study a setting that exhibits a Fano-type line shape, which is asymmetric and rapidly drops from $1$ to $0$.} \label{fig:lines}
	\end{figure}
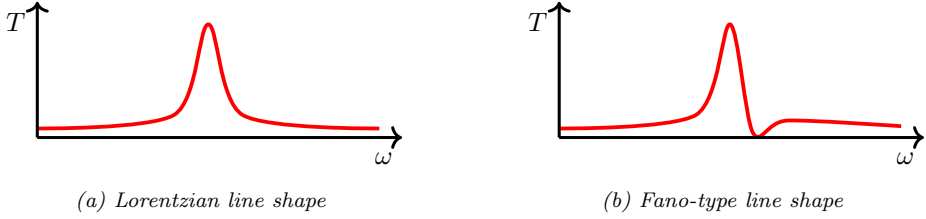

	\section{Metascreens} \label{sec:setup}
	We study a metascreen consisting of periodically repeated pairs of resonators, which are inclusions of a contrasting material surrounded by some background medium, as depicted in \Cref{fig:metascreen}. We will begin by presenting a capacitance matrix characterization of the band structure, similar to previous works \cite{ammari2020exceptional, ammari2017screen}. Thereafter, we will restrict our attention to the band structure in a neighbourhood of the origin and will compute the corresponding resonant frequencies explicitly. Using this analysis, we will compute the scattering matrix of the metascreen and demonstrate a Fano-type transmission asymmetry. In other words, for frequencies slightly below a critical frequency the transmission is close to 1, while for frequencies slightly above the critical point the transmission is close to 0. 
	
	We will study a structure composed of two resonators $D_1, D_2 \subset \R^3$ which are connected domains such that each boundary $\p D_i$ is Lipschitz continuous. The dimer $D$ is defined as $D= D_1\cup D_2$. We assume that the dimer is inversion symmetric in the sense that 
	\begin{equation} \label{inversionsymmetry}
	\P D_1 = D_2,
	\end{equation}
	where $\P:\R^3\to \R^3, \ \P(x) = -x$.
	
	We denote by $v_b$ the wave speed inside $D$ and by $v$ the wave speed in the surrounding material. We have that $v, v_b >0$ and, for simplicity, assume that the units are chosen such that $v=1$. Denoting the frequency of the waves by $\omega$, we define the wave numbers  as
	$$k = \frac{\omega}{v}, \quad k_b = \frac{\omega}{v_b}.$$	
	
	We will assume that there is a large material contrast between $D$ and the surrounding material, which is described by the contrast parameter $\delta$ as
	$$\delta \ll 1.$$	
		
	Next, we define the periodically repeated structure constituting the metascreen. We consider dimers in a two dimensional square lattice with period $L > 0$. The lattice is given by $\Lambda := L\Z^2$ with unit cell $Y=[-L/2,L/2]\times[-L/2,L/2]\times \R$. We assume that $D \Subset Y$ and  define the collection of periodically repeated resonators as
	$$\C = \bigcup_{(m_1,m_2)\in \Lambda} D + (m_1,m_2,0).$$
	This structure is depicted in \Cref{fig:metascreen}. The dual lattice $\Lambda^*$ of $\Lambda$ is defined as $\Lambda^* = (2\pi/L)\Lambda$.	The torus $Y^*:=\R^2/\Lambda^*$ is known as the \emph{Brillouin zone}. A function $f(y)$, $y \in \R^2$, is said to be $\alpha$-quasiperiodic, with \emph{quasiperiodicity} $\alpha\in Y^*$, if $e^{-\iu \alpha \cdot y}f(y)$ is periodic as a function of $y$. 
	
	We study the scattering problem 
	\begin{equation} \label{eq:scattering_quasi}
	\left\{
	\begin{array} {ll}
	\ds \Delta {u}+ \omega^2 {u}  = 0 & \text{in } \R^3 \setminus \C, \\[0.3em]
	\ds \Delta {u}+ \frac{\omega^2}{v_b^2}{u}  = 0 & \text{in } \C, \\
	\nm
	\ds  {u}|_{+} -{u}|_{-}  = 0  & \text{on } \partial \C, \\
	\nm
	\ds  \delta \frac{\partial {u}}{\partial \nu} \bigg|_{+} - \frac{\partial {u}}{\partial \nu} \bigg|_{-} = 0 & \text{on } \partial \C, \\
	\nm
	\ds u(x) - u^{\mathrm{in}}(x) & \text{satisfies the outgoing quasiperiodic} \\ & \text{radiation conditions as }  x_3 \rightarrow \pm\infty.
	\end{array}
	\right.
	\end{equation}
	Here, $u^{\mathrm{in}}$ is the incident field while  the subscripts $+$ and $-$ indicate the limits from outside and inside $D$, respectively.  We refer to, \textit{e.g.} \cite{bonnet1994guided, botten2013electromagnetic, ammari2020exceptional} for the definitions of the radiation conditions. We seek solutions $u$ which are $\alpha$-quasiperiodic in $(x_1,x_2)$ for some $\alpha$ in the sense that
	$$u(x + (m_1,m_2,0)) = e^{\iu\alpha\cdot (m_1,m_2)}u(x), \quad (m_1,m_2)\in \Lambda.$$
	If $u^{\mathrm{in}}$ is a plane wave $u^{\mathrm{in}}(x)= e^{\iu \k\cdot x}$, the quasiperiodicity $\alpha$ is specified by the wave vector $\k=\left(\begin{smallmatrix}
		k_1\\k_2\\k_3
	\end{smallmatrix}\right)$ as $\alpha = \left(\begin{smallmatrix}
	k_1 \\ k_2
\end{smallmatrix}\right)$ (see \emph{e.g.} \cite{ammari2018mathematical}). 
	
	If $u^{\mathrm{in}} = 0$, then frequencies $\omega$ with non-negative real part which are such that there is a nonzero solution $u$ for some $\alpha$ are known as \emph{(quasiperiodic) resonant frequencies} (or \emph{band functions} when viewed as functions of $\alpha$). The corresponding $\alpha$-quasiperiodic solutions $u$ are known as the \emph{(Bloch) eigenmodes} of the metascreen. A \emph{subwavelength} resonant frequency is a resonant frequency $\omega= \omega(\delta)$ which is continuous in $\delta$ and is such that $\omega(0) = 0$.
	
	We will study the scattering problem \eqref{eq:scattering_quasi} using a layer potential formulation.  For $\alpha\in Y^*$ such that $k \neq |\alpha+q|$ for all $q\in \Lambda^*$, the quasiperiodic Green's function $G^{\alpha,k}(x)$ is defined as the solution to
	$$\Delta G^{\alpha,k}(x) + k^2G^{\alpha,k}(x) = \sum_{(m_1,m_2) \in \Lambda} \delta(x-(m_1,m_2,0))e^{\iu \alpha\cdot (m_1,m_2)},$$
	along with the outgoing quasiperiodic radiation condition,  
	where $\delta(x)$ denotes the Dirac delta distribution. $G^{\alpha,k}$ can be written as
	\begin{equation}\label{eq:Gquasi}
	G^{\alpha,k}(x,y) := -\sum_{(m_1,m_2) \in \Lambda} \frac{e^{\iu k|x- (m_1,m_2,0)|}}{4\pi|x-(m_1,m_2,0)|}e^{\iu \alpha \cdot (m_1,m_2)}.
	\end{equation}
where the series converges uniformly for $x$ in compact sets of $\R^3$, $x\neq 0$ (see \emph{e.g} \cite[Section 2.12]{ammari2018mathematical}).
For $\varphi \in L^2(\p D)$ we define the quasiperiodic single layer potential $\S_D^{\alpha,k}$ by
	$$\S_D^{\alpha,k}[\varphi](x) := \int_{\partial D} G^{\alpha,k} (x-y) \varphi(y) \dx\sigma(y),\quad x\in \mathbb{R}^3.$$
	On the boundary of $D$, it satisfies the jump relations
	\begin{equation} \label{eq:jump1_quasi}
	\S_D^{\alpha,k}[\varphi]\big|_+ = \S_D^{\alpha,k}[\varphi]\big|_- \quad \mbox{on}~ \p D,
	\end{equation}
	and
	\begin{equation} \label{eq:jump2_quasi}
	\frac{\p}{\p\nu} \S_D^{\alpha,k}[\varphi] \Big|_{\pm} = \left( \pm \frac{1}{2} I +( \K_D^{-\alpha,k} )^*\right)[\varphi]\quad \mbox{on}~ \p D,
	\end{equation}
	where $(\K_D^{-\alpha,k})^*$ is the quasiperiodic Neumann--Poincar\'e operator, given by
	$$ (\K_D^{-\alpha, k} )^*[\varphi](x):= \mathrm{p.v.}\int_{\p D} \frac{\p}{\p\nu_x} G^{\alpha,k}(x-y) \varphi(y) \dx\sigma(y).$$
	
	We have the following result from \cite{ammari2020exceptional}.
	\begin{lemma} \label{lem:inv}
		The quasiperiodic single layer potential
		$\S_D^{\alpha,k} : L^2(\p D) \rightarrow H^1(\p D)$ is invertible if $k$ is small enough and $k \neq  |\alpha + q|$ for all $q \in \Lambda^*$.
	\end{lemma}

	\begin{figure}[tbh] 
		\begin{center}
			\includegraphics[width=0.95\linewidth]{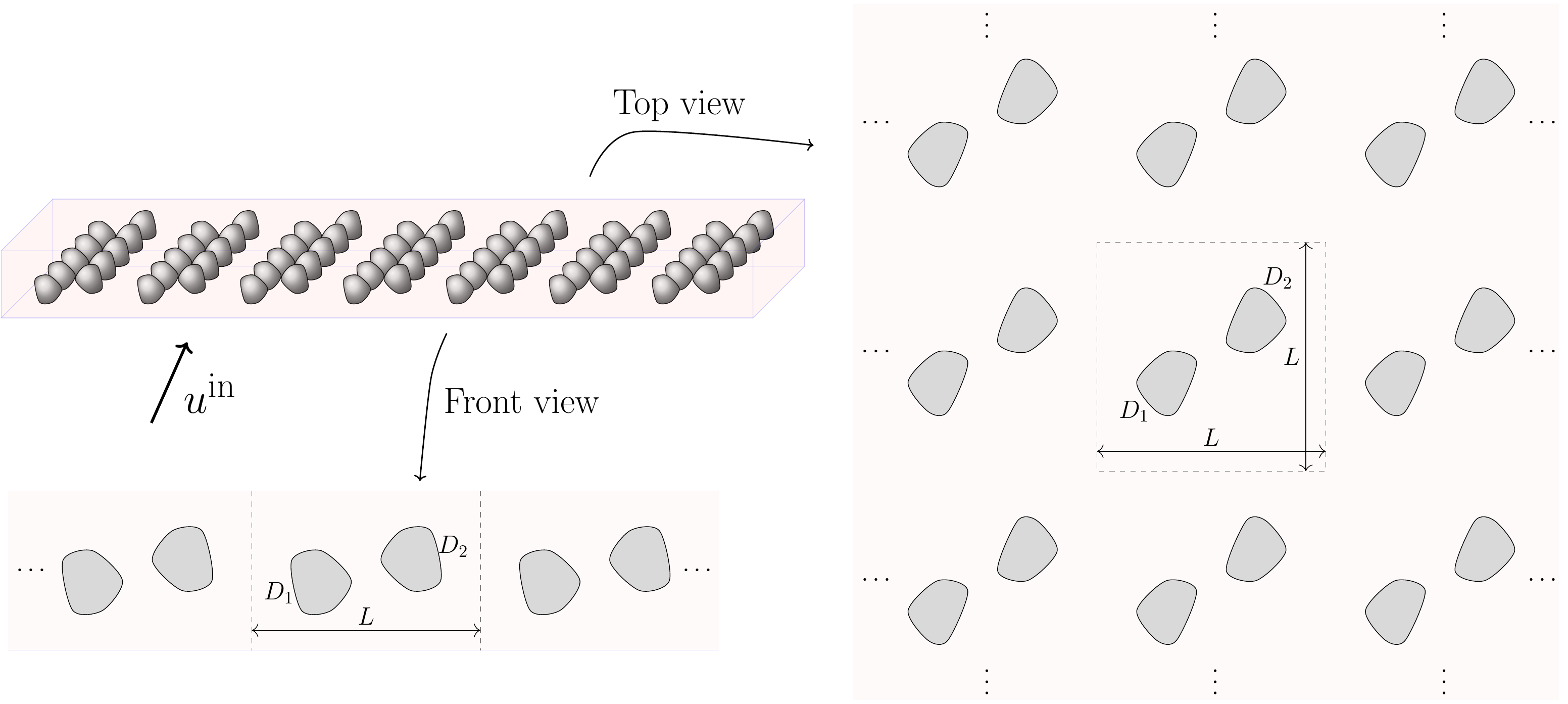}
		\end{center}
		\caption{In this work, we study a metascreen with an incident plane wave $u^{\mathrm{in}}$. The metascreen is composed of a $\mathcal{P}$-symmetric resonator dimer $D = D_1\cup D_2$ repeated periodically in a planar configuration.} \label{fig:metascreen}
	\end{figure}
	
	Recall that we are assuming $v=1$, so that $k=\omega$. The condition $\omega =  |\alpha+q|$ separates the $\omega\alpha$-plane into regimes with different radiation behaviour as $|x_3|\to \infty$. When $\omega$ is small, we have two regimes. In the regime where $\omega < \inf_{q\in \Lambda^*}|\alpha+q|$ (which is the unshaded region in \Cref{fig:band5000}) all waves are exponentially decaying as $|x_3|$ increases. The regime $|\alpha| < \omega < \inf_{q\in \Lambda^*\setminus \{0\} }|\alpha+q|$ (which is the shaded region in \Cref{fig:band5000}) corresponds to the first radiation continuum, where the waves typically behave as outgoing plane waves for large $x_3$.
	
	\subsection{Band structure} \label{sec:band}
	Here, we briefly mention the resonance problem $u^{\mathrm{in}} = 0$ in the regime when $\omega\to 0$ while $|\alpha| > c > 0$ for some $c$ independent of $\omega$ and $\delta$. In this regime, we have the  asymptotic expansions \cite{ammari2018mathematical}
	\begin{equation} \label{eq:expansions}
	\S_D^{\alpha,k} =\S_D^{\alpha,0} + O(k^2), \qquad	(\K_D^{-\alpha, k} )^* = (\K_D^{-\alpha, 0} )^* + O(k^2).
	\end{equation}
	Here, the error terms are stated with respect to the operator norms in the spaces $\B\big(L^2(\p D), H^1(\p D)\big)$ and $\B\big(L^2(\p D), L^2(\p D)\big)$, respectively, where $H^1(\p D)$ is the standard Sobolev space of functions that are square integrable and have a weak first derivative that is also square integrable. Furthermore, the error terms in \eqref{eq:expansions} are uniform for all $\alpha$ which satisfies $|\alpha| > c > 0$. For normed vector spaces $A$ and $B$, $\B(A,B)$ denotes the set of bounded linear operators from $A$ to $B$. We define the 
	\emph{quasiperiodic capacitance matrix} $C^\alpha =(C_{ij}^\alpha)_{i,j=1,2}$	 as
	\begin{equation}
	C_{ij}^\alpha=-\int_{\D_i} \psi_j^\alpha\de\sigma, \qquad \psi_j^\alpha=\left(\S_D^{\alpha,0}\right)^{-1}[\chi_{\p D_j}], 
	\end{equation}
	for $i,j=1,2$, where $\chi_{X}$ is used to denote the characteristic function of a set $X\subset\D$. From \textit{e.g.} \cite[Lemma 3.1]{ammari2020topologically} we know that $C^\alpha$ is a Hermitian matrix. The following result describes the subwavelength band structure \cite{ammari2020exceptional}.
	\begin{thm} \label{thm:res_quasi}
		Assume $|\alpha| > c > 0$ for some $c$ independent of $\delta$. As $\delta \rightarrow 0$, there are precisely two quasiperiodic resonant frequencies $\omega_1,\omega_2$ depending continuously on $\delta$ such that $\omega_i(0) = 0$. Moreover, they satisfy the asymptotic formula
		$$\omega_i = v_b \sqrt{\frac{\delta \lambda_i^\alpha}{|D_1|}} + O(\delta), \quad i = 1,2,$$
		where $|D_1|$ is the volume of a single resonator. Here, $\lambda_i^\alpha$ are the eigenvalues of the quasiperiodic capacitance matrix $C^{\alpha}$.
	\end{thm}

	\begin{figure} 
		\begin{center}
			\includegraphics[width=0.6\linewidth]{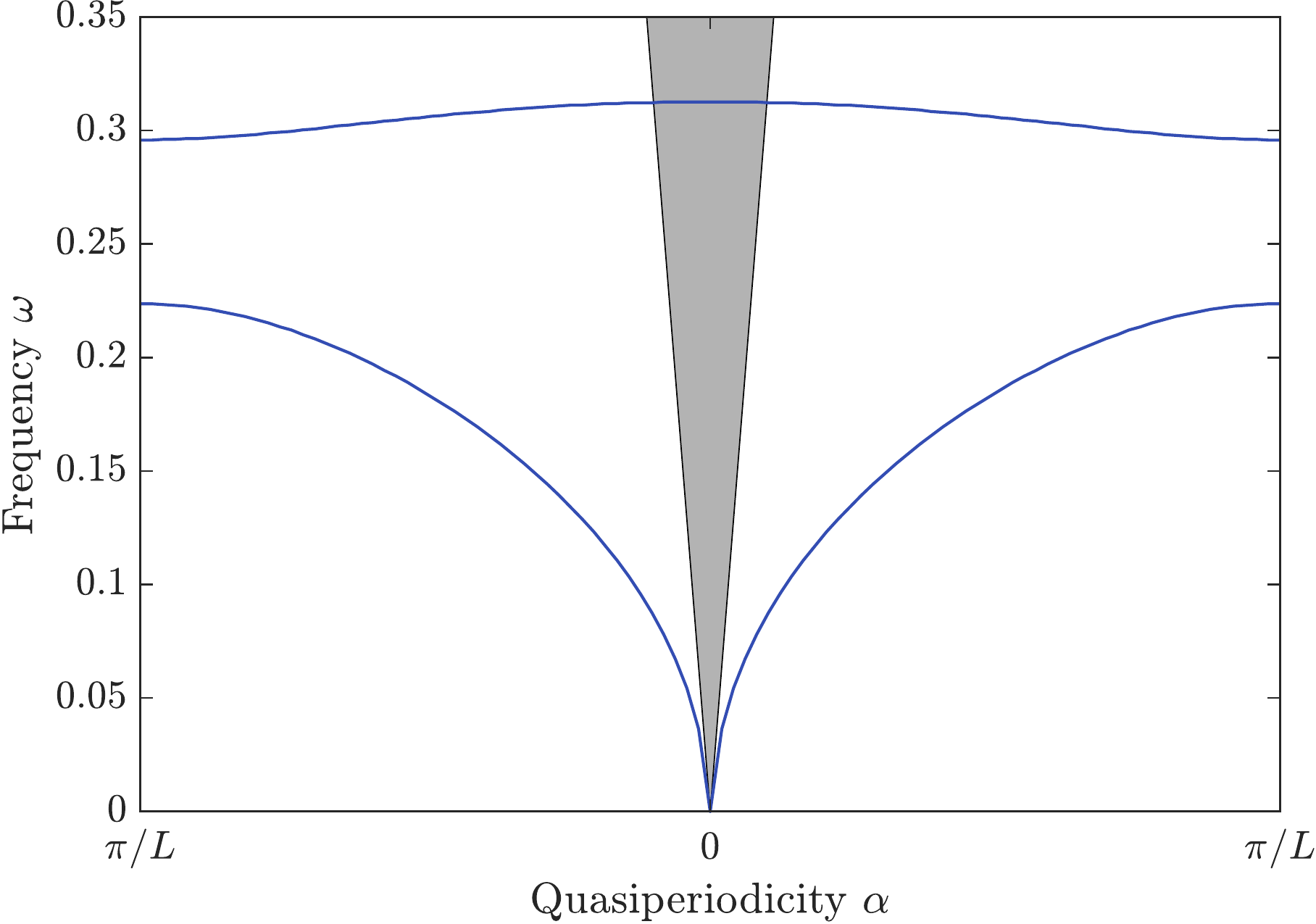}
		\end{center}
		\caption{The real part of the band structure, computed using the multipole expansion method. The shaded region is the first radiation continuum, defined by $|\alpha| < \omega < \inf_{q\in \Lambda^*\setminus \{0\} }|\alpha+q|$, while the unshaded region (apart from a neighbourhood of the origin) corresponds to the regime covered in \Cref{thm:res_quasi}. Here, we use the parameters $\delta=2\cdot 10^{-4}$ and $\theta = 0.05\pi$ (the same as in \Cref{fig:fanos3}).} \label{fig:band5000}
	\end{figure}
	We emphasize that \Cref{thm:res_quasi} holds also for a two-dimensional problem with a one-dimensional chain of resonators. In \Cref{fig:band5000} we plot the numerically computed subwavelength band structure (\textit{i.e.} $\omega_i$ as functions of $\alpha$) of such two-dimensional structure (here, we use the same parameters as in \Cref{fig:fanos3}, and we refer to \Cref{sec:num} for details on the setup and the method used). The shaded region shows the first radiation continuum, and will be the region of interest in the remainder of this work.

\section{Green's functions and capacitance matrix formulation}
In the analysis that follows, we will see that the Fano-type resonance occurs in a regime where both $\omega$ and $\alpha$ approach zero. In particular, we will study the case when the incident wave has a fixed direction of incidence and a frequency $\omega$ in the subwavelength regime. We define the wave vectors 
$$\k_+ = \begin{pmatrix}
	k_1 \\ k_2 \\  k_3
\end{pmatrix}, \quad \k_- = \begin{pmatrix}
	k_1 \\ k_2 \\  -k_3
\end{pmatrix}.$$
We will assume that the incident field  is a plane wave $u^{\mathrm{in}}(x) =
e^{\iu\k_+\cdot x}$. We consider subwavelength frequencies $\omega = O(\delta^{1/2})$ when $\delta \rightarrow 0$. In this limit, we assume that the incident direction is fixed, \ie{} that $\k_+$ is given by
\begin{equation} \label{eq:w}
	\k_+ = \omega\w, \ \text{for} \ \w = \begin{pmatrix}
	w_1 \\ w_2 \\  w_3
\end{pmatrix}\in\mathbb{R}^3, 
\end{equation}
where $\w$ is independent of $\omega$ and satisfies $|\w| = 1$ and $w_3 > 0$. We define
$$\alpha = \begin{pmatrix}k_1 \\ k_2 \end{pmatrix} = \omega \alpha_0 \in Y^*, \ \ \text{where} \ \ \alpha_0 = \begin{pmatrix} w_1 \\ w_2\end{pmatrix}.$$  In contrast to \Cref{sec:band}, this corresponds to the regime when $|\alpha| < k < \inf_{q\in \Lambda^*\setminus \{0\} }|\alpha+q|$. Related problems have been previously studied in \textit{e.g.} \cite{ammari2017screen,ammari2020exceptional}, and we begin by collecting some results on the Green's function in this setting.

	\subsection{Periodic Green's functions and capacitance matrix} \label{sec:C0}
	We begin by recalling some results from \cite{ammari2017screen,ammari2020exceptional}. When $|\alpha| < k < \inf_{q\in \Lambda^*\setminus \{0\} }|\alpha+q|$, the quasiperiodic Green's function admits the spectral representation
	\begin{equation} \label{eq:specrep} G^{\alpha,k}(x) = \frac{e^{\iu\alpha\cdot(x_1,x_2)}e^{\iu k_3|x_3|}}{2\iu k_3L^2} - \sum_{q\in \Lambda^*\setminus \{0\}}\frac{e^{\iu(\alpha+q)\cdot(x_1,x_2)}e^{-\sqrt{|\alpha+q|^2 - k^2}|x_3|}}{2L^2\sqrt{|\alpha+q|^2 - k^2} },\end{equation}
	where $k_3 = \sqrt{k^2 - |\alpha|^2}$. The series in \eqref{eq:specrep} converges uniformly for $x$ in compact sets of $\R^3$, $x\neq 0$,  and $|\alpha| < k < \inf_{q\in \Lambda^*\setminus \{0\} }|\alpha+q|$ (again, see \emph{e.g.} \cite{ammari2018mathematical}). In the case when $k = \alpha = 0$, we have
	\begin{equation} \label{eq:specrep_0}
	G^{0,0}(x) = \frac{|x_3|}{2L^2} - \sum_{q\in \Lambda^*\setminus \{0\}}\frac{e^{\iu q \cdot(x_1,x_2)}e^{-|q||x_3|}}{2L^2|q|}.
	\end{equation}
	Here  $G^{0,0}$ is called the \emph{periodic} Green's function \cite{ammari2017screen}. When $\omega \rightarrow 0$, we have
	\begin{equation} \label{eq:Gexp}
	G^{\omega\alpha_0,\omega}(x) = \frac{1}{2\iu \omega w_3 L^2} + G^{0,0}(x) + \frac{\alpha_0 \cdot (x_1,x_2)}{2w_3L^2} + \omega G_1^{\alpha_0} (x) + O(\omega^2).
	\end{equation}
	Here, $G_1^{\alpha_0} $ is a function independent of $\omega$, which can be written as \cite{ammari2017screen}
	$$
	G_1^{\alpha_0} (x) = \frac{\iu \left(w_3|x_3| + \alpha_0\cdot (x_1,x_2)\right)^2}{4w_3L^2} + \alpha_0 \cdot g_1(x),$$
	where $g_1(x)$ is a vector-valued function that is independent of $\alpha$ and $\omega$ and satisfies
	$$g_1(x_1,x_2,x_3) = g_1(x_1,x_2,-x_3), \qquad g_1(x_1,x_2,x_3) = - g_1(-x_1,-x_2,x_3).$$
	From \eqref{eq:Gexp} we in particular observe that the Green's function has a singularity of order $\omega^{-1}$.	We define the operators $\hat{\S}_D^{\alpha,k}: L^2(\p D) \rightarrow H^1(\p D)$ and $(\hat\K_D^{-\alpha,k})^*: L^2(\p D) \rightarrow L^2(\p D)$ as
	\begin{equation}
	\label{eq:Shat}\hat{\S}_D^{\alpha,k}[\varphi](x) = {\S}_D^{0,0}[\varphi](x) - \frac{\iu - \alpha\cdot(x_1,x_2)}{2\iu k_3 L^2} \int_{\p D}\varphi \dx \sigma - \int_{\p D}\frac{\alpha \cdot (y_1,y_2)}{2k_3L^2} \varphi(y)\dx \sigma (y),
	\end{equation}
	and
	$$(\hat\K_D^{-\alpha,k})^*[\varphi](x) = (\K_D^{0,0})^*[\varphi](x) +  \frac{\alpha \cdot (\nu_{x,1}, \nu_{x,2})}{2k_3L^2}\int_{\p D} \varphi \dx \sigma.$$
	Here, $\nu_x = (\nu_{x,1}, \nu_{x,2}, \nu_{x,3})$ denotes the outwards pointing normal of $D$ at $x$.
	Moreover, we define the operators $\S_1^{\alpha_0} : L^2(\p D) \rightarrow H^1(\p D)$ and  $(\K_{D,1}^{-\alpha_0})^*: L^2(\p D) \rightarrow L^2(\p D)$ as
	$$\S_1^{\alpha_0}[\varphi](x) := \int_{\partial D} G_1^{\alpha_0} (x-y) \varphi(y) \dx\sigma(y), \qquad (\K_{D,1}^{-\alpha_0})^*[\varphi](x):= \int_{\p D} \frac{\p}{\p\nu_x} G_1^{\alpha_0}(x-y) \varphi(y) \dx\sigma(y).$$
	We then have the asymptotic expansion \cite{ammari2020exceptional}
	\begin{equation} \label{eq:exp_hat}
	\S_D^{\omega\alpha_0,\omega} = \hat{\S}_D^{\omega\alpha_0,\omega} + \omega\S_1^{\alpha_0} + O(\omega^2), \qquad (\K_D^{-\omega\alpha_0,\omega})^* = (\hat\K_D^{-\omega\alpha_0,\omega})^* +  \omega(\K_{D,1}^{-\alpha_0})^* + O(\omega^2),
	\end{equation}
	as $\omega \rightarrow 0$, where the error terms are with respect to corresponding operator norms. We have the next three results from \cite{ammari2020exceptional}.
	\begin{lemma} \label{lem:intK}
		For any $\varphi\in L^2(\p D)$ we have, for $i=1,2$,
	$$\int_{\p D_i}\left(-\frac{1}{2}I+(\hat\K_D^{-\alpha,k})^*\right)[\varphi]\de\sigma=0, \qquad \int_{\p D_i}(\K_{D,1}^{-\alpha_0})^*[\varphi] = \frac{\iu|D_i|}{2w_3L^2}\int_{\p D}\varphi \de\sigma.$$
	\end{lemma}

\begin{lemma} \label{lem:ker}
	The dimension of $\ker \S_D^{0,0}$ is at most one. Further, if $\varphi \in L^2(\p D)$ is such that $\int_{\p D} \varphi \dx \sigma = 0$ and $\S_D^{0,0}[\varphi] = K\chi_{\p D}$ for some constant $K$, then $\varphi = 0$.
\end{lemma}

\begin{lemma} \label{lem:holo}
	For any $\alpha_0 \in Y^*$ with $|\alpha_0| < 1$,  $\left(\hat{\S}_D^{\omega\alpha_0,\omega}\right)^{-1}$ and $\left({\S}_D^{\omega\alpha_0,\omega}\right)^{-1}$ are   holomorphic operator-valued functions of $\omega$ in a neighbourhood of $\omega = 0$.
\end{lemma}

Let $L^2_0(\p D)$ be the mean-zero space defined as $$L^2_0(\p D) = \left\{f \in L^2(\p D) \ \left| \ \int_{\p D} f \dx \sigma = 0 \right. \right\}.$$
Then $\S_D^{0,0}$ is invertible from $L^2_0(\p D)$ onto its image, which does not contain the constant functions.

We will now define the analogous capacitance coefficients in the periodic setting. Since $({\S}_D^{\omega\alpha_0,\omega})^{-1}$ is a holomorphic function of $\omega$ we have, as $\omega \to 0$,
	$$
	\left({\S}_D^{\omega\alpha_0,\omega}\right)^{-1} = \S_0^{\alpha_0} + \omega \S_{-1}^{\alpha_0} +  O(\omega^2),$$
	with respect to the operator norm in $\B\big(H^1(\p D), L^2(\p D)\big)$, for some operators $\S_0^{\alpha_0}$ and $\S_{-1}^{\alpha_0}$ which are independent of $\omega$. For  $\alpha_0$ with $|\alpha_0| < 1$, we let
	\begin{equation} \label{eq:psi0}
		\psi_i^{0} = \S_0^{\alpha_0}[\chi_{\p D_i}], \qquad {\psi}_i^{1,\alpha_0} = \S_{-1}^{\alpha_0}[\chi_{\p D_i}],
	\end{equation}
	and then define the periodic capacitance coefficients as
	\begin{equation}
	C_{ij}^{0}=-\int_{\D_i} \psi_j^{0}\de\sigma, \quad i,j = 1,2.
	\end{equation}
	We call the matrix $C^{0}  = (C_{ij}^{0})$ the \emph{periodic capacitance matrix}. Although the definition of the periodic capacitance matrix depends on $\alpha_0$, we will later see that $\psi_j^{0}$ and $C^{0}$ are independent of $\alpha_0$ in the current setting. First, we have from \cite{ammari2020exceptional} the following result concerning the periodic capacitance coefficients. 
	\begin{lemma} \label{lem:cap0}
		The periodic capacitance matrix $C^{0}$ is a real matrix given by		
		\begin{equation*}
		C^{0} = C_{11}^{0}\begin{pmatrix} 1 & -1 \\ -1 & 1\end{pmatrix}.
		\end{equation*}
		\end{lemma}
	In fact, from \cite{ammari2020exceptional} we have that $\psi^0_1 = -\psi_2^0$. We also define the ``higher-order'' coefficients
	$$C_{ij}^{1,\alpha_0} = -\int_{\D_i} {\psi}_j^{1,\alpha_0} \dx \sigma, \quad  \mathbf{c}_i=\int_{\D} y\psi_i^{{0}}(y)\de\sigma(y), \quad i,j = 1,2.$$
	Then $\c_1 = - \c_2$ and we write the vector $\c_1$ as 
	\begin{equation}\label{eq:c1}
		\c_1 = \begin{pmatrix}c_1\\c_2\\c_3\end{pmatrix}.
	\end{equation} 
	We then have a result to describe the capacitance coefficients (which is a development of a similar symmetry result proved in \cite[Lemma 3.17]{ammari2020exceptional}).
	\begin{lemma} \label{lem:C1}
	It holds that
	\begin{itemize}
		\item[(i)]$\psi_j^0$, and consequently $C_{ij}^0$ and $\mathbf{c}_j$, are independent of $\alpha_0$. 
		\item[(ii)]
		$\ds C^{1,\alpha_0} = -\frac{\iu w_3 L^2}{2} \begin{pmatrix} 1 & 1 \\ 1 & 1\end{pmatrix} + \iu  (\alpha_0,0)\cdot\c_1\begin{pmatrix} 0 & 1 \\ -1 & 0\end{pmatrix} -\frac{\iu w_3 c_3^2}{2L^2} \begin{pmatrix} 1 & -1 \\ -1 & 1\end{pmatrix} +  O(\omega),\\[0.3em]$
		where $w_3$ and $c_3$ are defined in \eqref{eq:w} and \eqref{eq:c1}, respectively.
	\end{itemize}
	\end{lemma}	
\begin{proof}[Proof of (i)]
	To emphasise the role of $\alpha_0$, we will use the notation $\psi_i^0 = \psi_i^{0,\alpha_0}$ and $\c_i = \c_i^{\alpha_0}$ in this proof. From \Cref{lem:holo} we have the following expansion
	$$
		\left(\hat{\S}_D^{\alpha,\omega}\right)^{-1}[\chi_{\p D_i}] = \psi_i^{0,\alpha_0} + \omega\hat{\psi}_i^{1,\alpha_0} + O(\omega^2),
	$$
	for some function $\hat{\psi}_i^{1,\alpha_0}$ and for $i=1,2$. Expanding the orders of $\omega$, we find that 
	\begin{align}
		\int_{\p D} \psi_j^{0,\alpha_0}  \dx \sigma &= 0, \label{eq:-1th} \\
		\S_D^{0,0}[\psi_j^{0,\alpha_0}] + \frac{1}{2\iu w_3 L^2}\int_{\p D}\hat{\psi}_j^{1,\alpha_0} \dx \sigma - \int_{\p D}\frac{\alpha_0 \cdot (y_1,y_2)}{2w_3L^2} \psi_j^{0,\alpha_0}(y)\dx \sigma (y) &= \chi_{\p D_j}, \label{eq:0th} \\
		\S_D^{0,0}[\hat{\psi}_j^{1,\alpha_0}] + \frac{\alpha_0\cdot(x_1,x_2)}{2 w_3 L^2} \int_{\p D}\hat{\psi}_j^{1,\alpha_0} \dx \sigma + K\chi_{\p D} &= 0, \label{eq:1st}
	\end{align}
	for some constant $K$. Under the symmetry assumption \eqref{inversionsymmetry}, we have	
	\begin{equation}\label{eq:sym}
		{\psi}_1^{0,\alpha_0}(y) = {\psi}_2^{0,-\alpha_0}(\P y) =  -{\psi}_1^{0,-\alpha_0}(\P y), \qquad \hat{\psi}_1^{1,\alpha_0}(y) = \hat{\psi}_2^{1,-\alpha_0}(\P y),
	\end{equation}
	and, in particular, $\c_i^{\alpha_0} = \c_i^{-\alpha_0}$. Then, from \eqref{eq:0th} we find that
	\begin{equation}\label{eq:first}
		\S_D^{0,0}[\psi_1^{0,\alpha_0}+\psi_2^{0,-\alpha_0}] + \frac{1}{\iu w_3 L^2}\int_{\p D}\hat{\psi}_1^{1,\alpha_0}\dx \sigma -\frac{(\alpha_0,0) \cdot \c_1^{\alpha_0}}{w_3L^2} = \chi_{\p D},
	\end{equation}
	and that	
	\begin{equation}\label{eq:second}
	\S_D^{0,0}[\psi_1^{0,\alpha_0} - \psi_2^{0,-\alpha_0}] =\chi_{\p D_1}- \chi_{\p D_2}.
	\end{equation}
	In other words, for some constant $K$ we have
	\begin{equation}\label{eq:sumpsi}
		\S_D^{0,0}[\psi_1^{0,\alpha_0}+\psi_2^{0,-\alpha_0}]  = K\chi_{\p D},
	\end{equation}
	and since $\int_{\p D}\psi_i^{0,\alpha_0} = 0$, we have from \Cref{lem:ker} that $\psi_1^{0,\alpha_0} = -\psi_2^{0,-\alpha_0}$. Then, from \eqref{eq:second} we find that 
	$$\S_D^{0,0}[\psi_1^{0,\alpha_0}] =\frac{1}{2}\chi_{\p D_1}- \frac{1}{2}\chi_{\p D_2}.$$
	Since $\S_D^{0,0}$ and the right-hand side are independent of $\alpha_0$, and since $\S_D^{0,0}$ is injective on $L^2_0(\p D)$, we find that $\psi_i^{0,\alpha_0}$ is independent of $\alpha_0$.
	\end{proof}
	\begin{proof}[Proof of (ii)]
		From \cite[Lemma 3.11]{ammari2020exceptional} we have
		$$\psi_j^{1,\alpha_0} = \hat{\psi}_j^{1,\alpha_0} - \left(\hat\S_D^{\omega\alpha_0,\omega}\right)^{-1}\S_1^{\alpha_0} \psi_j^{0},$$
		and from \eqref{eq:-1th} it follows that 
		\begin{equation} \label{eq:C1}
			C^{1,\alpha_0} = \hat{C}^{1,\alpha_0} + h\begin{pmatrix} 1 & -1 \\ -1 & 1\end{pmatrix}, \qquad h = \int_{\D}\S_1^{\alpha_0}[\psi_1^{0}] \psi_1^{0} \dx \sigma,			
		\end{equation}
		where $\hat{C}^{1,\alpha_0} = \left(\hat{C}_{ij}^{1,\alpha_0}\right)_{i,j=1,2}$ is the matrix given by
		\begin{equation}\label{eq:Chat}
			\hat{C}_{ij}^{1,\alpha_0} = -\int_{\p D_i} \hat{\psi}_j^{1,\alpha_0} \dx \sigma.
		\end{equation}
		We begin by computing $\hat{C}^{1,\alpha_0}$. From \eqref{eq:first} we find that, for $j=1,2$,
	\begin{equation}\label{eq:firstsys}
	\int_{\p D} \hat{\psi}_j^{1,\alpha_0} \dx \sigma = \iu w_3L^2 + \iu(\alpha_0,0) \cdot \c_j.
	\end{equation}
	Multiplying \eqref{eq:1st} by $\psi_i^{0}$ and integrating around $\p D$ we have, using \eqref{eq:-1th}, that
	$$\int_{\p D} \psi_i^0 \S_D^{0,0}[\hat{\psi}_j^{1,\alpha_0}] \dx \sigma + \frac{(\alpha_0,0)\cdot\c_i}{2 w_3 L^2} \int_{\p D}\hat{\psi}_j^{1,\alpha_0} \dx \sigma = 0.$$
	Since $\S_D^{0,0}$ is self-adjoint in $L^2(\p D)$, we find using \eqref{eq:0th} that
	$$\int_{\p D_i} \hat{\psi}_j^{1,\alpha_0} \dx \sigma -  \frac{1}{2\iu w_3 L^2}\left( \int_{\p D}\hat{\psi}_i^{1,\alpha_0} \dx \sigma -2\iu  (\alpha_0,0)\cdot\c_i \right) \int_{\p D} \hat{\psi}_j^{1,\alpha_0} \dx \sigma = 0.$$
	Together with \eqref{eq:firstsys}, we find that
	$$\int_{\p D_i} \hat{\psi}_j^{1,\alpha_0} \dx \sigma =  \frac{1}{2\iu w_3 L^2}\left(\iu w_3L^2 - \iu(\alpha_0,0) \cdot \c_i\right)\left(\iu w_3L^2 + \iu(\alpha_0,0) \cdot \c_j\right).$$
	From the definition of $\hat{C}^{1,\alpha_0}$ in \eqref{eq:Chat}, we then have 
	\begin{equation} \label{eq:C1hat}
		\hat{C}^{1,\alpha_0} = -\frac{\iu w_3 L^2}{2} \begin{pmatrix} 1 & 1 \\ 1 & 1\end{pmatrix} + \iu  (\alpha_0,0)\cdot\c_1\begin{pmatrix} 0 & 1 \\ -1 & 0\end{pmatrix} -\frac{\big( (\alpha_0,0)\cdot\c_1\big)^2}{2\iu w_3L^2} \begin{pmatrix} 1 & -1 \\ -1 & 1\end{pmatrix} +  O(\omega).
	\end{equation}
	The only remaining task is to explicitly compute $h$. We follow the proof of \cite[Lemma 3.17]{ammari2020exceptional} and write the kernel function $G_1^{\alpha_0}$ of $\S_1^{\alpha_0}$ as
	$$G_1^{\alpha_0}(x) = K_1(x) + K_2^{\alpha_0}(x) + K_3^{\alpha_0}(x),$$
	where
	$$K_1(x) = \frac{\iu w_3x_3^2}{4L^2}, \quad K_2^{\alpha_0}(x) = \alpha_0\cdot\left( \frac{\iu|x_3|(x_1,x_2)}{2L^2} + g_1(x)\right), \quad K_3^{\alpha_0}(x) = \frac{\iu \left(\alpha_0\cdot (x_1,x_2)\right)^2}{4w_3L^2}.$$
	We have 
	$$K_1(x-y) = \frac{\iu w_3}{4L^2}\left(x_3^2 - 2x_3y_3 + y_3^2\right),$$
	and since $\int_{\D}\psi_1^0\dx\sigma = 0$ we conclude that 
	\begin{align*}
		I_1:=\int_{\D}\int_{\D} K_1(x-y)\psi_1^0(x)\psi_1^0(y) \dx\sigma(x)\dx\sigma(y) &= -\frac{\iu w_3}{2L^2}\int_{\D}x_3 \psi_1^0(x)\dx\sigma(x)\int_{\D}y_3 \psi_1^0(y)\dx\sigma(y) \\
		&=-\frac{\iu w_3 c_3^2}{2L^2}.
	\end{align*}
	We observe that $K_2^{\alpha_0}(\P x) = - K_2^{\alpha_0}(x)$ while $\psi_1^0(\P x) = -\psi_1^0(x)$. Therefore 
	$$I_2:=\int_{\D}\int_{\D} K_2^{\alpha_0}(x-y)\psi_1^0(x)\psi_1^0(y) \dx\sigma(x)\dx\sigma(y) = 0.$$
	Finally, we have
	$$K_3^{\alpha_0}(x-y) = \frac{\iu}{4w_3L^2}\left( \left(\alpha_0\cdot (x_1,x_2)\right)^2 - 2\left(\alpha_0\cdot (x_1,x_2)\right)\left(\alpha_0\cdot (y_1,y_2)\right) + \left(\alpha_0\cdot (y_1,y_2)\right)^2 \right),$$
	and hence
	\begin{multline*}
		I_3:=\int_{\D}\int_{\D} K_3^{\alpha_0}(x-y)\psi_1^0(x)\psi_1^0(y) \dx\sigma(x)\dx\sigma(y) = \\ \frac{\iu}{4w_3L^2}\left(\int_{\D}\left(\alpha_0\cdot (x_1,x_2)\right)^2 \psi_1^0(x)\dx \sigma(x) \int_{\D}\psi_1^0\dx\sigma\right. 
		+\int_{\D}\left(\alpha_0\cdot (y_1,y_2)\right)^2 \psi_1^0(y)\dx \sigma(y) \int_{\D}\psi_1^0\dx\sigma\\ -2\left.\int_{\D}\alpha_0\cdot(x_1,x_2) \psi_1^0(x)\dx\sigma(x)\int_{\D}\alpha_0\cdot(y_1,y_2) \psi_1^0(y)\dx\sigma(y)\right) \\
		= \frac{\big( (\alpha_0,0)\cdot\c_1\big)^2}{2\iu w_3L^2}.
	\end{multline*}
	In total, we see that
	$$h = I_1+I_2+I_3=-\frac{\iu w_3 c_3^2}{2L^2} + \frac{\big( (\alpha_0,0)\cdot\c_1\big)^2}{2\iu w_3L^2}.$$
	This, together with \eqref{eq:C1} and \eqref{eq:C1hat}, proves the claim.
	\end{proof}
	
	In order to introduce more sophisticated symmetry assumptions, we define the maps $\P_{12}:\R^3\to \R^3$ and $\P_3:\R^3\to \R^3$ by	
	$$\P_{12}(x_1,x_2,x_3) = (-x_1,-x_2,x_3) \quad\text{and}\quad \P_3(x_1,x_2,x_3) = (x_1,x_2,-x_3).$$
	We will occasionally use $\P, \P_{12}$ and $\P_3$ as operators on $L^2(\p D)$ defined through composition, \textit{e.g.} for $\varphi \in L^2(\p D)$ we define $(\P\varphi)(x) = \varphi(\P x)$. The following lemma, which shows properties of $\c_1$ defined in \eqref{eq:c1}, follows directly from symmetry arguments.
	\begin{lemma} \label{lem:Psym}
		If $\P_3 D_i = D_i$ for $i=1,2$, then $c_3 =0$. If instead $\P_3 D_1 = D_2$, then $c_1 = c_2 = 0$.
\end{lemma}	
We remark that if $\P_3 D_1 = D_1$ it follows from \eqref{inversionsymmetry} that $\P_3 D_2 = D_2$.

\subsection{Green's function for the free-space Helmholtz equation}
We conclude this section by collecting some well-known results on the Green's function in free-space $G^k$, given by 
$$
G^k(x) = 
-\frac{e^{\iu k|x|}}{4\pi|x|}.$$
These results will be useful for representing the solution of \eqref{eq:scattering_quasi} in the interior of $D$ (for more details we refer, for example, to \cite{ammari2018mathematical}). For a given bounded domain $D$ in  $\mathbb{R}^{3}$, with Lipschitz boundary, the single layer potential of the density function $\varphi\in L^{2}(\partial D)$ is defined by
\begin{equation*}
	\S_{D}^k[\varphi](x):=\int_{\partial D} G^k( x-y ) \varphi(y)\dx \sigma(y),~~~x\in \mathbb{R}^{3}.
\end{equation*}
Then  the following jump relation holds
\begin{equation} \label{jump1}
	\left. \frac{\partial}{\partial\nu} \S_{D}^k[\varphi]\right|_{\pm}(x)=\left (\pm\frac{1}{2}I+(\K_{D}^{k})^*\right)[\varphi](x),~~~x\in \partial D,
\end{equation}
where the operator $(\K_{D}^{k})^*$ is the Neumann--Poincar\'e operator associated to the domain $D$ and is defined by 
\begin{equation*}
	(\K_{D}^{k})^*[\varphi](x)= \mathrm{p.v.}\int_{\p D}\frac{\partial G^k(x - y)}{\partial \nu(x)} \varphi(y)\dx \sigma(y),~~~x\in \partial D.
\end{equation*}
We denote $\S_D := \S_D^0$ and $\K_D^* := (\K_D^{0})^*$. For a small  $k$ we have asymptotic expansions given by \cite[Appendix A]{ammari2018minnaert}
\begin{align}
	\S_{D}^{k} [\varphi]=\S_D[\varphi]+  \sum_{n=1}^\infty k^n \S_{D,n} [\varphi], \quad
	(\K_{D}^{k})^*[\varphi] =\K_{ D}^{*}+ \sum_{n=1}^\infty k^n \K_{D,n} [\varphi], \label{eq:expSK}
\end{align}
which converge in $\B(L^2(\p D), H^1(\p D)))$ and $\B(L^2(\p D), L^2(\p D)))$, respectively, where
\begin{align}
	\S_{D,n} [\varphi] (x)& := -\frac{\iu^n}{4\pi n!}  \int_{\p D} |x - y|^{n-1} \varphi(y) \dx \sigma(y), \quad n=1,2,\dots,\\
	\K_{D,n} [\varphi] (x)& := -\frac{\iu^n (n-1)}{4\pi n!}  \int_{\p D} {\langle\, x -y, \nu_x \rangle }{|x - y|^{n-3}}\varphi(y) \dx \sigma(y), \quad n=1,2,\dots.
\end{align}
It is known that $\S_D : L^2(\p D)\rightarrow H^1(\p D)$ is invertible and that its inverse is bounded.

Finally, we present some useful formulas which are frequently used in the subsequent analysis and were proved in \cite{ammari2017double}. 
\begin{lemma}\label{lem:trick_layer}
	The following identities hold for any $\varphi\in L^2(\p D)$: for $j = 1, 2$,
	\begin{align*}
		(i)& \quad \int_{\p D_j}   \left(-\frac{1}{2}I+\K_{ D}^{*} \right) [\varphi]  \dx \sigma =0, & (ii)& \quad \int_{\p D_j} \K_{D,2} [\varphi] \dx \sigma= -\int_{D_j}\S_D [\varphi]\dx \sigma, \\
		(iii)& \quad \int_{\p D_j} \K_{D,3}[\varphi]\dx \sigma= \frac{\iu |D_j|}{4\pi} \int_{\p D} \varphi \dx \sigma, & (iv)& \quad \S_{D,1} [\varphi] (x) = -\frac{\iu}{4\pi}  \int_{\p D}\varphi \dx \sigma.
	\end{align*}
\end{lemma}

	\section{Characterization of Fano-type resonances}
	 \label{sec:fano}
	 In this section we will study the resonant frequencies of \eqref{eq:scattering_quasi}. The solutions to \eqref{eq:scattering_quasi} can be represented as
\begin{equation} \label{eq:rep_quasi}
u = \begin{cases} u^\mathrm{in} +
\S_{D}^{\alpha,\omega}[\psi](x), & x\in Y \setminus \overline{D},\\
\S_{D}^{\frac{\omega}{v_b}}[\phi](x), & x\in \D,
\end{cases}
\end{equation} 
for some surface densities $(\phi,\psi)\in L^2(\D)\times L^2(\D)$.
Using the jump conditions \eqref{eq:jump1_quasi} and \eqref{eq:jump2_quasi}, we see that the layer densities $\phi$ and $\psi$ satisfy
\begin{align}
\S_{D}^{\frac{\omega}{v_b}}[\phi]-\S_{D}^{\alpha,\omega}[\psi]=u^{\mathrm{in}} \quad \text{on}\ \D, \label{eq:quasi_full1} \\
\left(-\frac{1}{2}I+(\K_{D}^{{\frac{\omega}{v_b}}})^*\right)[\phi]-\delta\left(\frac{1}{2}I+(\K_{D}^{-\alpha,\omega})^*\right)[\psi]=\delta\ddp{u^{\mathrm{in}}}{\nu} \quad \text{on} \ \D. \label{eq:quasi_full2}
\end{align}
Letting $\eta = \S_{D}^{\omega\alpha_0,\omega}[\psi]$, the equations \eqref{eq:quasi_full1} and \eqref{eq:quasi_full2} can be written equivalently as
\begin{equation}\label{int_eq_single}
\mathcal{A}^\omega [\eta] =  F[u^{\mathrm{in}}],
\end{equation}
where 
\begin{equation}
\mathcal{A}^\omega :=\left( -\frac{1}{2} I +(\K_D^{\frac{\omega}{v_b}})^* \right) \left(\S_D^{\frac{\omega}{v_b}}\right)^{-1}- \delta \left( \frac{1}{2} I + (\K_D^{-\omega\alpha_0,\omega})^*\right)
\left(\S_D^{\omega\alpha_0,\omega}\right)^{-1},
\end{equation}
and 
$$F[u^{\mathrm{in}}] := \delta\ddp{u^{\mathrm{in}}}{\nu}- \left(- \frac{1}{2} I + (\K_D^{\frac{\omega}{v_b}})^*\right)\left(\S_D^{\frac{\omega}{v_b}}\right)^{-1}[u^{\mathrm{in}}].$$
Then $\mathcal{A}^\omega$ is holomorphic in a neighbourhood of the origin. Moreover, the quasiperiodic resonant frequencies are precisely the characteristic values of $\mathcal{A}^\omega$ with non-negative real part; in other words the values of $\omega$ in the right-half plane such that $\A^\omega[\eta]=0$ has a non-trivial solution $\eta$ (see, for example, \cite[Chapter 1]{ammari2018mathematical} for the definition and further properties of characteristic values).

We now investigate the subwavelength resonant frequencies that exist in the current setting. By analogous steps as those in the proof of \cite[Lemma 3.1]{ammari2020honeycomb}, we obtain the following result. 
\begin{lemma}
	Let $\alpha = \omega\alpha_0$ for some $\alpha_0$ which is independent of $\delta$ and satisfies $|\alpha_0|=1$. Then, there are precisely two quasiperiodic resonant frequencies $\omega_1$ and $\omega_2$ which depend continuously on $\delta$ and are such that $\omega_i(0) = 0$. 
\end{lemma}

Given their existence, we may derive the asymptotic behaviour of $\omega_i$ as $\delta\to0$, which will be used to describe the Fano-type resonance of the metascreen.
\begin{thm} \label{thm:res0} 
	Let $\alpha = \omega\alpha_0$ for some $\alpha_0$ that is independent of $\delta$ and satisfies $|\alpha_0|=1$. Let $w_3$ and $c_3$ be defined by \eqref{eq:w} and \eqref{eq:c1} respectively. Then, as $\delta \to 0$, 	
\begin{align*}
\omega_1 &= - \frac{\iu \delta v_b^2 w_3 L^2 }{|D_1|} + O(\delta^2), \\
\omega_2& = \sqrt{\frac{2\delta v_b^2 C_{11}^0}{|D_1|}} - \frac{\iu \delta v_b^2  w_3 c_3^2}{2|D_1|L^2} + O(\delta^{3/2}).
\end{align*}
\end{thm}
\begin{proof}
	We will expand the operator $\A^\omega$ as $\delta \to0$ and $\omega = O(\delta^{1/2})$. Since
\begin{multline*} 
\left(\S_D^{\frac{\omega}{v_b}}\right)^{-1}=\S_D^{-1} - \frac{\omega}{v_b}\S_D^{-1}\S_{D,1}\S_D^{-1} +\frac{\omega^2}{v_b^2}\Big(\S_D^{-1}\left(\S_{D,1}\S_D^{-1}\right)^2-\S_D^{-1}\S_{D,2}\S_D^{-1}\Big)\\
+\frac{\omega^3}{v_b^3}\Big(\S_D^{-1}\S_{D,1}\S_D^{-1}\left(\S_{D,2}-\S_{D,1}\S_D^{-1}\S_{D,1}\right)\S_D^{-1}+\S_D^{-1}\left(\S_{D,2}\S_D^{-1}\S_{D,1}-\S_{D,3}\right)\S_D^{-1}\Big)+ O(\omega^4)
 \end{multline*}
and $\left(-\frac{1}{2}I+\K_D^*\right)\S_D^{-1}\S_{D,1}=0$,  $\mathcal{A}^\omega$ can be expanded as
\begin{multline*}
\mathcal{A}^\omega = \left(-\frac{1}{2}I+\K_D^*\right)\S_D^{-1} +\frac{\omega^2}{v_b^2} \left( \K_{D,2}\S_D^{-1}-  \left(-\frac{1}{2}I+\K_D^*\right)\S_D^{-1} \S_{D,2}\S_D^{-1}\right)\\
+\frac{\omega^3}{v_b^3} \left(\K_{D,3}\S_D^{-1} -\K_{D,2}\S_D^{-1} \S_{D,1}\S_D^{-1}+\left(-\frac{1}{2}I+\K_D^*\right)\left(\S_D^{-1}\left(\S_{D,2}\S_D^{-1}\S_{D,1}-\S_{D,3}\right)\S_D^{-1}\right)\right)\\-\delta \left(\frac{1}{2}I+(\hat\K_D^{-\omega\alpha_0,\omega})^* + \omega (\K_{D,1}^{-\alpha_0})^*\right) \left(\S_D^{\omega\alpha_0,\omega}\right)^{-1} +O(\delta\omega^2+\omega^4).
\end{multline*}
Suppose $\mathcal{A}^\omega[\eta] =0$.  Then $\eta$ can be written as
\begin{equation} \label{eq:psi_basis_quasi}
\eta=q_1\chi_{\partial D_1}+q_2\chi_{\partial D_2}+O(\omega^2+\delta),
\end{equation}
for constants $q_1,q_2$ with $|q_1|+|q_2|=O(1)$. Integrating $\mathcal{A}^\omega[\eta] =0$ over $\partial D_i$, we have
\begin{align*}
-\frac{\omega^2}{v_b^2} |D_i| \delta_{i1} q_1+ \delta\left( C_{i1}^0+ {\omega} C_{i1}^{1,\alpha_0} \right)q_1 
-\frac{\omega^2}{v_b^2} |D_i| \delta_{i2} q_2   + \delta\left( C_{i2}^0+ {\omega} C_{i2}^{1,\alpha_0} \right)q_2  =O(\delta \omega^2+\omega^4).
\end{align*}
Thus it reduces to the problem 
\begin{equation}\label{eq:eval_C_quasi}
\left(\delta {C}^{0} + {\delta\omega}C^{1,\alpha_0} -\frac{\omega^2|D_1|}{v_b^2}I  \right)\begin{pmatrix}q_1\\q_2\end{pmatrix}
= O(\delta \omega^2+\omega^4).
\end{equation}
We observe that the eigenvalues $\lambda_i$ of $C^0 + \omega C^{1,\alpha_0}$ satisfy, as $\omega\to 0$, 
$$\lambda_i = \lambda_{0i} + \omega \vb_{0i}^\mathrm{T} C^{1,\alpha_0} \vb_{0i} + O(\omega^2),$$
where $\lambda_{0i},\vb_{0i}$ is an eigenpair of $C^0$. In other words, 
$$\lambda_1 = -\iu \omega w_3 L^2 + O(\omega^2), \qquad \lambda_2 = 2C_{11}^0 -\frac{\iu\omega w_3 c_3^2}{L^2}  + O(\omega^2).$$
Thus, we see that either
\begin{align*}
2\delta C_{11}^0-\frac{\iu\delta\omega w_3 c_3^2}{L^2} - \frac{\omega^2|D_1|}{v_b^2}  = O(\delta\omega^2 +\omega^4)
\end{align*}
or
\begin{align*}
-\iu  \delta\omega w_3 L^2- \frac{\omega^2|D_1|}{v_b^2} 
= O(\delta\omega^2 +\omega^4).
\end{align*}
Therefore, we have
\begin{align*}
\omega_1 &= - \frac{\iu \delta v_b^2 w_3 L^2 }{|D_1|} + O(\delta^2), \\
\omega_2& = \sqrt{\frac{2\delta v_b^2 C_{11}^0}{|D_1|}} - \frac{\iu \delta v_b^2  w_3 c_3^2}{2|D_1|L^2} + O(\delta^{3/2}),
\end{align*}
which concludes the proof.
\end{proof}

\begin{remark}
	We have proved that the imaginary part of $\omega_1$ scales like $\delta$, while the imaginary part of $\omega_2$ scales like $c_3^2\delta$, where $c_3$ was defined in \eqref{eq:c1}. From \Cref{lem:Psym}, we know that we can make $c_3$ very small by choosing a structure which is almost symmetric under $\P_3$. In particular, if the dimers are aligned with an angle $\theta$ with the $x_1x_2$-plane, then we have that $c_3 \to0$ as $\theta\to 0$. Therefore $\omega_1$ corresponds to a broad resonance peak while $\omega_2$ corresponds to a sharp peak in the transmission spectrum, and \Cref{thm:res0} characterizes the Fano-type resonance. As we will see in the following sections, this will generate a Fano-type transmission anomaly when $c_3$ is very small.
\end{remark}

\section{Metascreen scattering}
We now assume that $u^{\mathrm{in}}(x) = e^{\iu\k_+\cdot x}$ and seek the behaviour of the solution $u$ of \eqref{eq:scattering_quasi}. Throughout the remainder of this work, we will write $ f\sim g$ to denote that two continuous functions $f,g\in C(\R)$ are equal up to exponentially decaying factors, in the sense that there is some constant $K>0$ such that
$$|f(x) - g(x)| = O(e^{-Kx}) \text{ as } x\rightarrow \infty.$$
In the first radiation continuum, the scattered field $u-u^\mathrm{in}$ consists of a single propagating mode as $|x_3| \rightarrow \infty$. Then the total field $u$, which is the solution to \eqref{eq:scattering_quasi}, will behave as 
\begin{equation}\label{eq:utot}
	u \sim \begin{cases} t e^{\iu\k_+\cdot x}, & x_3\rightarrow \infty, \\ e^{\iu\k_+\cdot x} + r e^{\iu\k_-\cdot x}, & x_3\rightarrow -\infty.\end{cases}
\end{equation}
The coefficients $r$ and $t$ are the reflection and transmission coefficients, while the scattering matrix $S = S(\omega)$ is given by
$$S(\omega) = \begin{pmatrix} r & t \\ t & r \end{pmatrix}.$$

Since we are using layer potential techniques to solve the scattering problem \eqref{eq:scattering_quasi}, the next result will be our main tool to compute the radiative behaviour in the far field \cite{ammari2020exceptional}.
\begin{lemma} \label{prop:radiation}
	Assume that $|\alpha| < k < \inf_{q\in \Lambda^*\setminus \{0\} }|\alpha+q|$. Then, as $|x_3|\rightarrow \infty$, the quasiperiodic single layer potential satisfies 
	$$
	\S_D^{\alpha,k}[\varphi] \sim \begin{cases}\ds \frac{e^{\iu\k_+\cdot x}}{2\iu k_3L^2}\int_{\D} e^{-\iu\k_+\cdot y}\varphi(y)\dx \sigma(y), \quad &x_3\rightarrow \infty, \\[2em]
		\ds \frac{e^{\iu\k_-\cdot x}}{2\iu k_3L^2}\int_{\D} e^{-\iu\k_-\cdot y}\varphi(y)\dx \sigma(y), \quad &x_3\rightarrow -\infty,
	\end{cases}
	$$
	where $k_3 = \sqrt{k^2 - |\alpha|^2}$ while $\k_\pm = \left(\begin{smallmatrix}\alpha_1\\ \alpha_2\\\pm k_3\end{smallmatrix}\right)$ and $\alpha = \left(\begin{smallmatrix}\alpha_1\\ \alpha_2\end{smallmatrix}\right).$ 
\end{lemma}
For later reference, we have the following asymptotic behaviour as $\omega\to0$ \cite{ammari2020exceptional}
\begin{equation}\label{eq:uinR}
	\frac{1}{2 \iu k_3L^2}\int_{\D} e^{-\iu \k_\pm\cdot y}\left(\S_D^{\omega\alpha_0,\omega}\right)^{-1}[u^{\mathrm{in}}](y)\dx \sigma(y) = 1+O(\omega),
\end{equation}
which can be used to simplify the expressions from \Cref{prop:radiation}.

\section{Embedded eigenvalues and bound states in the continuum}
\Cref{thm:res0} shows that the $O(\delta)$-imaginary part of $\omega_2$ vanishes when the structure is symmetric. In fact, we will prove in this section that the second resonance $\omega_2$ becomes exactly real under the additional assumption of perpendicular incidence. We will characterize the eigenmodes corresponding to this real resonance and show that they are bound states in the continuum. Specifically, we show that they do not radiate to the far field and, reciprocally, cannot be excited from the far field.

Throughout this section, in addition to the assumption of inversion symmetry $\P D_1=D_2$ from \eqref{inversionsymmetry} that is imposed on the resonator dimer, we will assume that each individual resonator is symmetric in the sense that $\P_3 D_i = D_i$ for $i=1,2$. For example, \Cref{fig:2D} with $\theta = 0$ satisfies this condition. In addition, we will assume that $\alpha_0 = 0$, which corresponds to the case that the incident waves are perpendicular to the metascreen.
\begin{prop} \label{prop:realeigen}
	Assume that $\P_3 D_i = D_i$ for $i=1,2$ and that $\alpha_0=0$. Then $\omega_2$ is real.
\end{prop}
\begin{proof}
We define $\C: L^2(\p D) \to L^2(\p D)$ as 
$$\C[\varphi] = -\chi_{\p D_1}\int_{\p D_1}\S_0^{\alpha_0}[\varphi]\dx \sigma - \chi_{\p D_2}\int_{\p D_2}\S_0^{\alpha_0}[\varphi]\dx \sigma,$$
and then we define
$$\A_0^\omega = \left(-\frac{1}{2}I+\K_D^*\right)\S_D^{-1} -\frac{\omega^2|D_1|}{v_b^2} I  + \delta \C.$$
It is clear that $\A_0^\omega$ has a characteristic value given by
$$\omega_{0} = \sqrt{\frac{2\delta v_b^2 C_{11}^0}{|D_1|}},$$
and that $\eta_0 = \frac{1}{\sqrt{2|D_1|}}\left(\chi_{\p D_1}-\chi_{\p D_2}\right)$ spans the kernels
$$\ker\left(\A_0^{\omega_{0}}\right) = \mathrm{span}\left\{\eta_{0}\right\} \quad\text{and}\quad \ker\left((\A_0^{\omega_{0}})^*\right) = \mathrm{span}\left\{\eta_{0}\right\}.$$
By arguments analogous to those used in \textit{e.g.} \cite{ammari2020honeycomb,ammari2019encapsulated}, we have the following pole-pencil decomposition for $\omega$ close to $\omega_0$,
$$\left(\A_0^\omega\right)^{-1} = \frac{L}{\omega-\omega_0} + R(\omega), \qquad L = \frac{\langle \cdot,\eta_0 \rangle\eta_0 }{\langle \frac{\mathrm{d}}{\mathrm{d}\omega}\A_0^{\omega_0}[\eta_0],\eta_0\rangle},$$
where $\langle \cdot, \cdot\rangle$ denotes the $L^2(\p D)$-inner product and $R(\omega)$ is a holomorphic function of $\omega$ in a neighbourhood of $\omega_{0}$ satisfying 
$$R(\omega)[\eta_0] = r(\omega)\eta_0,$$
for some function $r(\omega)$ which is real-valued for real $\omega$. We have 
$$\left\langle \frac{\mathrm{d}}{\mathrm{d}\omega}\A_0^{\omega_0}[\eta_0],\eta_0\right\rangle =- \frac{2\omega_0|D_1|}{v_b^2}.$$
By the characteristic value perturbation theory, as in the proof of \cite[Theorem 3.9]{ammari2009layer}, we have that 
$$\omega_2-\omega_{0} = \frac{\mbox{tr}}{2\pi \iu}\sum_{p=1}^\infty \frac{1}{p}\int_{\p V} \left((\A_0^\omega)^{-1} (\A_0^\omega-\A^\omega)   \right)^p \dx\omega.$$
Using the property that $\mbox{tr}\int_{\p V} AB \dx \omega = \mbox{tr}\int_{\p V}BA\dx \omega$ for finitely meromorphic operators $A$ and $B$ in a neighbourhood of $\omega_0$ \cite[Proposition 1.7]{ammari2018mathematical}, we have
\begin{align}\label{eq:pert}
	\omega_2-\omega_{0} &= \frac{\mbox{tr}}{2\pi \iu}\sum_{p=1}^\infty\sum_{q=1}^{p-1} \frac{1}{p}\binom{p}{q}\int_{\p V}\frac{  L^{p-q}R^q}{(\omega-\omega_0)^{p-q}}  \dx\omega \nonumber \\
	&= \frac{1}{2\pi \iu}\sum_{p=1}^\infty\sum_{q=0}^{p-1} \frac{(-1)^{p-q}}{p}\binom{p}{q}\left(\frac{v_b}{2\omega_0|D_1|}\right)^{p-q} \int_{\p V}\frac{ r(\omega)^q  \langle (\A_0^\omega-\A^\omega)[\eta_0],\eta_0 \rangle^{p-q}}{(\omega-\omega_0)^{p-q}}  \dx\omega \nonumber \\
	&= \sum_{p=1}^\infty\sum_{q=0}^{p-1} \frac{1}{(p-q-1)!p}\binom{p}{q}\left(\frac{v_b}{2\omega_0|D_1|}\right)^{p-q} \frac{\mathrm{d}^{p-q-1}}{\mathrm{d} \omega ^{p-q-1}} \Big(r(\omega)^q  \langle \A^\omega[\eta_0],\eta_0 \rangle^{p-q}\Big)\Big|_{\omega=\omega_0}.
\end{align}
Next, we will show that, under the symmetry assumptions $\P_3 D_i = D_i$ and $\alpha_0=0$, the factor $\langle \A^\omega[\eta_0],\eta_0 \rangle$ is real for all $\omega\in\R$. This, together with \eqref{eq:pert}, shows that $\omega_2$ is real.

We write $\Im\langle \A^\omega[\eta_0],\eta_0 \rangle = \I_1 + \I_2$, where 
\begin{align*}
	\I_1 &= \Im\left\langle \left( -\frac{1}{2} I +(\K_D^{\frac{\omega}{v_b}})^* \right) \left(\S_D^{\frac{\omega}{v_b}}\right)^{-1}[\eta_0],\eta_0\right\rangle, \\
	\I_2 &= - \delta \Im\left\langle \left( \frac{1}{2} I + (\K_D^{-\omega\alpha_0,\omega})^*\right)
	\left(\S_D^{\omega\alpha_0,\omega}\right)^{-1} [\eta_0],\eta_0\right\rangle.
\end{align*}
We begin by studying $\I_1$. Let $v$ be the solution to
$$
\left\{
\begin{array} {ll}
	\ds \Delta {v}+ \frac{\omega^2}{v_b^2} {v}  = 0 & \text{in } D, \\
	\nm
	\ds v  = \eta_0  & \text{on } \partial D.
\end{array}
\right.
$$
Then $\tfrac{\p v}{\p \nu}$ is real-valued and it follows that 
$$\left\langle \left( -\frac{1}{2} I +(\K_D^{\frac{\omega}{v_b}})^* \right) \left(\S_D^{\frac{\omega}{v_b}}\right)^{-1}[\eta_0],\eta_0\right\rangle = \int_{\p D_1}\frac{\p v}{\p \nu} \dx \sigma - \int_{\p D_2}\frac{\p v}{\p \nu} \dx \sigma$$
is real-valued. Hence $\I_1 = 0$.

We next turn to $\I_2$. We observe that 
$$\overline{G^{0,k}}(x) = G^{0,k}(x) - \frac{\cos\left( k|x_3|\right)}{\iu kL^2}.$$
We define $\overline{\S_D^{0,\omega}}$ and $\overline{(\K_D^{0,\omega})^*}$ as the operators corresponding to the kernel $\overline{G^{0,k}}$. Then, for functions $\varphi$ which are such that $\varphi(\P_{12} x) = -\varphi(x)$ we have that
\begin{equation}\label{eq:symSK}
	\overline{\S_D^{0,\omega}}[\varphi]= {\S_D^{0,\omega}}[\varphi], \qquad \overline{(\K_D^{0,\omega})^*}[\varphi] = {(\K_D^{0,\omega})^*}[\varphi].
\end{equation}
Denote $\psi = \left(\S_D^{0,\omega}\right)^{-1}[\eta_0]$. Since
\begin{equation}\label{eq:Ssym}
	\S_D^{0,\omega}[\P_{12}\psi](x) = \S_D^{0,\omega}[\psi](\P_{12} x) = -\S_D^{0,\omega}[\psi](x),
\end{equation}
we conclude that $\psi(\P_{12} x) = -\psi(x)$. Therefore, using \eqref{eq:symSK} and since $\eta_0$ is real-valued, we have that 
$$\eta_0 = \overline{\S_D^{0,\omega}}[\overline{\psi}] = {\S_D^{0,\omega}}[\overline{\psi}],$$
which means that $\overline{\psi} = \psi$. Again using \eqref{eq:symSK}, we therefore have
$$\overline{\left( \frac{1}{2} I + (\K_D^{0,\omega})^*\right)[\psi]} = \left( \frac{1}{2} I + (\K_D^{0,\omega})^*\right)[\psi],$$
from which we can see that
$$\mathcal{I}_2=0.$$

To conclude, we have proved that $\langle \A^\omega[\eta_0],\eta_0 \rangle$ is real for all real $\omega$. This, together with \eqref{eq:pert}, proves that $\omega_2$ is real.
\end{proof}	
Next, we study the eigenmodes at $\omega_2$, and demonstrate that they correspond to bound states in the continuum. We first show that the eigenmodes at $\omega_2$ are exponentially decaying functions of $x_3$, meaning that they do not radiate energy into the far field. 
\begin{prop}\label{prop:rad0}
Assume that $\P_3 D_i = D_i$ for $i=1,2$ and that $\alpha_0=0$. Let $u$ be a solution to \eqref{eq:scattering_quasi} with $u^\mathrm{in}=0$ and $\omega=\omega_2$. For sufficiently small $\delta$, we have as $x_3\to \pm \infty$ that
$$u \sim 0.$$
\end{prop}
\begin{proof}
	Let $\A^{\omega_2}[\eta] = 0$. From the $\P_3$-symmetry follows that $\P_{12}\eta$ is in the kernel of $\A^{\omega_2}$ at $-\alpha_0$. Since $\alpha_0=0$, and since $\omega_2$ is a simple characteristic value, we find that
	$$\P_{12}\eta = K\eta$$
	for some constant $K$. For small $\delta$ we know that $\eta = \chi_{\p D_1} - \chi_{\p D_2} + O(\delta)$. We conclude that $K  = -1$, or in other words that $\eta$ is odd under $\P_{12}$ (we emphasize that this holds exactly, \textit{i.e.} that also the $O(\delta)$-term of $\eta$ is odd). Then, as $x_3\rightarrow \pm \infty$, we have
	$$u \sim \frac{e^{\iu\k_\pm\cdot x}}{2\iu k_3L^2}\int_{\D} e^{-\iu\k_\pm\cdot y}\left(\S_D^{0,\omega}\right)^{-1}[\eta](y)\dx \sigma(x).$$
	From the symmetry, we have that $\left(\S_D^{0,\omega}\right)^{-1}[e^{-\iu\k_\pm\cdot y}]$ is an odd function under $\P_{12}$ (as was the case in \eqref{eq:Ssym}). Since $\alpha_0 =0$, the functions $e^{-\iu\k_\pm\cdot y}$ are even, so $u\sim0$ which proves the claim.
\end{proof}

The next result (which is the reciprocal result of \Cref{prop:rad0}) shows that there will be no transmission peak at $\omega = \omega_2$ (corresponding to the fact that the bound state in the continuum cannot be excited from the far field). We remark, however, that there might be some small but nonzero transmission $t=O(\delta^{1/2})$ originating from the first, broad resonance.
\begin{prop}\label{prop:trans0}
	Assume that $\P_3 D_i = D_i$ for $i=1,2$ and that $\alpha_0=0$. Then, at $\omega = \omega_2$, the scattering matrix $S(\omega)$ satisfies 
	$$S(\omega_2)= -\begin{pmatrix} 1  & 0 \\ 0 & 1\end{pmatrix} + O(\delta^{1/2}).$$	
\end{prop}
\begin{proof}
	Let $\A^{\omega_2}[\eta] = F[u^\mathrm{in}]$. Even though the equation is not uniquely solvable, the scattering matrix is well-defined since (by \Cref{prop:rad0}) any $\eta_{\mathrm{k}} \in \ker\A^{\omega_2}$ corresponds to an exponentially decaying part of the solution. We can decompose the solution $\eta$  as 
	$$\eta = \eta_0 + \eta_{\mathrm{k}}-u^\mathrm{in},$$
	where $\P_{12}\eta_0 = \eta_0$ and $\eta_{\mathrm{k}}$ is an element in $\ker\A^{\omega_2}$. Then, as in the proof of \Cref{thm:res0} we can show that
	$$\eta_0 = K\chi_{\p D} + O(\delta),$$
	for some constant $K$. From \eqref{eq:uinR} we have 
	\begin{align} \label{eq:radsym}
		u-u^{\mathrm{in}} &\sim \frac{e^{\iu\k_\pm\cdot x}}{2\iu k_3L^2}\int_{\D} e^{-\iu\k_\pm\cdot y}\left(\S_D^{0,\omega}\right)^{-1}[\eta](y)\dx \sigma(y) \nonumber \\
		&\sim \left(K-1+O(\delta^{1/2})\right)e^{\iu\k_\pm\cdot x}.
	\end{align}
	As in the proof of \Cref{thm:res0} (\textit{cf.} also the proof of \Cref{prop:sol} below), we find from $\A^{\omega_2} [\eta_0 - u^\mathrm{in}] = F[u^\mathrm{in}]$ that 
		$$K\left({C}^{0}  - \lambda_2^0 I  \right)\begin{pmatrix}1\\1\end{pmatrix}
	= O(\delta^{1/2}),
	$$	
	and since $\left(\begin{smallmatrix} 1\\1 \end{smallmatrix}\right)$ is not in the kernel of $(C^0-\lambda_2^0I)$ we find that $K = O(\delta^{1/2})$. The expression for $S$ then follows from \eqref{eq:radsym}.
\end{proof}

\begin{remark}
	It is enlightening to compare and contrast the behaviour at real resonances in the present case to the non-Hermitian case studied in \cite{ammari2020exceptional}. In the present case, \Cref{prop:rad0} and \Cref{prop:trans0} show that the corresponding eigenmodes are bound states in the continuum. Consequently, there is no transmission peak at the resonant frequency. Similarly to the present case, the resonances may also become real in the non-Hermitian case. However, in this case it is possible that the corresponding modes indeed couple to the far field. This gives a singularity of the transmitted field at the resonant frequency, corresponding to so-called extraordinary transmission \cite{ammari2020exceptional}.
\end{remark}

\begin{remark}
	 When the symmetry is broken, the real eigenvalue $\omega_2$ will be shifted into the complex plane and the corresponding mode will be coupled to the far field. We emphasize that the symmetry can be broken in two distinct fashions: either by making $\alpha_0 \neq 0$ or by perturbing the $\P_3$-symmetry of $D$ to make $c_3$ nonzero (\textit{e.g.} by choosing $\theta \neq 0$ in \Cref{fig:2D}). In order to achieve Fano-type transmission anomalies, we should design the system so that the two resonances interfere; in particular, we should make the imaginary part of $\omega_1$ rather large. In view of \Cref{thm:res0}, we observe that we want $w_3$ large. Therefore, we chose to break the symmetry by making $c_3$ nonzero, and in the next section we will compute the scattering matrix and demonstrate the Fano-type transmission anomaly.
\end{remark}
	
	\section{Fano-type transmission anomaly}
	In this section we compute the scattering matrix of the metascreen. The goal is to demonstrate an asymmetric transmission peak around the second resonance $\omega_2$, which is characteristic of a Fano-type transmission anomaly.
	
	We will begin with a characterization of the solution to the scattering problem \eqref{eq:scattering_quasi}. We define the functions 
	$$S_j^{\alpha,\omega}(x) = \begin{cases}
		\S_{D}^{\alpha,\omega}[\psi_j^{0} + \omega \psi_j^{1,\alpha_0}](x), & x\in \R^3 \setminus \overline\C,\\[0.3em]
		\S_{D}^{\frac{\omega}{v_b}}[\psi_j](x), & x\in \C,
	\end{cases}
	$$
	where $\psi_j = \S_D^{-1}[\chi_{\p D_j}]$. We then have the following result (which generalizes \cite[Proposition 3.14]{ammari2020exceptional} to the present setting).
	\begin{prop}\label{prop:sol}
		Assume that $c_3 \neq 0$ and that $\omega$ is real with $0\leq\omega \leq K\sqrt{\delta}$ for some constant $K>0$. Then, as $\delta \to 0$,
		\begin{equation}\label{eq:scattered}
		u - u^{\mathrm{in}} = q_1S_1^{\alpha,\omega} + q_2 S_2^{\alpha,\omega} - \S_D^{\alpha,\omega}\left(\S_D^{\alpha,\omega}\right)^{-1}[u^{\mathrm{in}}] + O(\omega^2),
	\end{equation}
		where $q:=\left(\begin{smallmatrix}
			q_1 \\ q_2
		\end{smallmatrix}\right)$ satisfies the problem
		$$\left({C}^{0} + {\omega}C^{1,\alpha_0} -\frac{\omega^2|D_1|}{\delta v_b^2}I  \right)\begin{pmatrix}q_1\\q_2\end{pmatrix}
		= -\begin{pmatrix}p_1\\p_2\end{pmatrix} +O(\delta), \qquad p_i = \int_{\D_i}\left({\S}_D^{\omega\alpha_0,\omega}\right)^{-1}[u^{\mathrm{in}}]\de\sigma.
		$$	
	\end{prop}
	\begin{proof}
		 We solve the equation 
		\begin{equation}
			\label{eq:eq}
			\mathcal{A}^\omega [\eta] =  \delta\ddp{u^{\mathrm{in}}}{\nu}- \left(- \frac{1}{2} I + (\K_D^{\frac{\omega}{v_b}})^*\right)\left(\S_D^{\frac{\omega}{v_b}}\right)^{-1}[u^{\mathrm{in}}].
		\end{equation}
	When $c_3 \neq 0$ and $\omega$ is real, we know that $\mathcal{A}^\omega$ is invertible with bounded inverse. Following the proof of \Cref{thm:res0}, we have that $\eta$ satisfies
		$$
		\eta=q_1\chi_{\partial D_1}+q_2\chi_{\partial D_2}-u^{\mathrm{in}}+O(\omega^2+\delta),
		$$
		which proves \eqref{eq:scattered}. To prove the equation for $q$, we proceed as in the proof of \Cref{thm:res0} by expanding $\A^\omega [\eta]$ and integrating around $\p D_i$, to get
		\begin{align*}
			\int_{\p D_i} \mathcal{A}^\omega[\eta]\dx \sigma &= -q_i|D_1|  \frac{\omega^2}{v_b^2} -\delta \int_{\p D_i} \left(\S_D^{\omega\alpha_0,\omega}\right)^{-1}\eta\dx \sigma  +\int_{D_i}  \frac{\omega^2}{v_b^2}u^{\mathrm{in}}\dx x +O(\delta\omega^2+\omega^4).
		\end{align*}
		Turning to the right-hand side of \eqref{eq:eq}, we have
		\begin{align*}
			\int_{\p D_i}F[u^{\mathrm{in}}]\dx \sigma &=  - \int_{\p D_i}\left(- \frac{1}{2} I + (\K_D^{\frac{\omega}{v_b}})^*\right)\left(\S_D^{\frac{\omega}{v_b}}\right)^{-1}[u^{\mathrm{in}}]\dx \sigma + O(\delta^2) \\
			&= -\int_{\p D_i}\frac{\p }{\p \nu} \left(\S_D^{\frac{\omega}{v_b}}\right)^{-1}[u^{\mathrm{in}}]\dx \sigma + O(\delta^2) \\
			&= \frac{\omega^2}{v_b^2}\int_{D_i}u^{\mathrm{in}}\dx x+ O(\delta^2).
		\end{align*}
		Therefore, we have the equation
		$$-q_i|D_1|  \frac{\omega^2}{v_b^2} -\delta\int_{\p D_i} \left(\S_D^{\omega\alpha_0,\omega}\right)^{-1}\eta = O(\delta\omega^2 + \omega^4).$$
		Defining $ p_i = \int_{\D_i}\left({\S}_D^{\omega\alpha_0,\omega}\right)^{-1}[u^{\mathrm{in}}]\de\sigma$, the above equation can be written in matrix form as
		$$\left(\delta{C}^{0} + {\omega\delta}C^{1,\alpha_0} -\frac{\omega^2|D_1|}{\delta v_b^2}I  \right)\begin{pmatrix}q_1\\q_2\end{pmatrix}
		= -\delta\begin{pmatrix}p_1\\p_2\end{pmatrix} +O(\delta\omega^2 +\omega^4),
		$$	
		which proves the claim.
	\end{proof}
\subsection{Scattering matrix and Fano-type transmission anomaly}\label{sec:Sfano}
	We are now able to compute the scattering matrix and demonstrate the Fano-type asymmetric transmission line. We define the coefficients	
	$$R_{j,\pm} = \frac{1}{2 \iu k_3L^2}\int_{\D} e^{-\iu \k_\pm\cdot y}\left(\psi_j^{0}(y) + \omega\psi_j^{1,\alpha_0}(y)\right)\dx \sigma(y), \quad i = 1,2,$$
	which describe the radiation of $S_j^{\alpha,\omega}$ as $x\to \pm \infty$. As $\omega \rightarrow 0$, we have the following asymptotic behaviour
	\begin{align}\label{eq:R}
		R_{j,\pm} &= \frac{1}{2\iu k_3L^2}\left(\int_{\D}\psi_j^{0}(y)\dx \sigma(y) - \int_{\D} \iu \k_\pm\cdot y\psi_j^{0}(y)\dx \sigma(y) + \omega\int_{\D} \psi_j^{1,\alpha_0}\dx \sigma\right) +  O(\omega) \nonumber\\
		&= - \frac{\k_\pm \cdot \mathbf{c}_j}{2 k_3L^2} + \frac{1}{2 \iu w_3L^2}\int_{\D} \psi_j^{1,\alpha_0}\dx \sigma  + O(\omega) = \frac12 -\frac{1}{2 k_3L^2}\big( \k_\pm - (\alpha,0)\big)\cdot \c_j + O(\omega) \nonumber\\
		&=	\frac12 \pm\frac{(-1)^j c_3}{2 L^2} + O(\omega).
	\end{align}
	We then have the following main theorem.
	\begin{thm} \label{thm:phano}
		Assume that $c_3 \neq 0$ and that $0\leq \omega \leq K\sqrt{\delta}$ for some constant $K$. Then we have the following asymptotic expansion of the scattering matrix as $\delta \to 0$
		\begin{equation}\label{eq:scattering}		 
		 S = \frac{\omega_1}{\omega_1-\omega}\begin{pmatrix} 1&1\\1&1\end{pmatrix} +\frac{2\iu\omega\Im(\omega_2)}{\omega_2^2-\omega^2}\begin{pmatrix} 1&-1\\-1&1\end{pmatrix} - \begin{pmatrix} 1&0\\0&1\end{pmatrix} + O(\delta^{1/2}),		 
	 \end{equation}	
	where the error term is uniform with respect to $\omega$.
	\end{thm}
	
	\begin{proof}
		We begin by computing $p$. Recall that $u^{\mathrm{in}}(x) =
		e^{\iu\k_+\cdot x}$. We then have
		\begin{align*} \label{eq:q}
			p_i=\int_{\D_i}\left({\S}_D^{\alpha,\omega}\right)^{-1}[u^{\mathrm{in}}]\de\sigma &= \int_{\D}u^{\mathrm{in}}\left({\S}_D^{-\alpha,\omega}\right)^{-1}[\chi_{\p D_i}]\de\sigma \\&= \int_{\D}\iu \k_+\cdot x \psi_i^{0} \de\sigma + \omega\int_{\D} \psi_i^{1,-\alpha_0}\de\sigma + O(\omega^2) \nonumber\\ &= \iu k_3 L^2 +\iu\big(\k_+-(\alpha,0) \big)\cdot \c_i + O(\delta),
		\end{align*}
	so that 
	$$p= \iu k_3 L^2\begin{pmatrix}1\\1\end{pmatrix}+\iu k_3c_3\begin{pmatrix}1\\-1\end{pmatrix} + O(\delta).$$
	As $x\to \pm \infty$, we have using \eqref{eq:R} and \eqref{eq:uinR} that
	\begin{align*}
	u-u^{\mathrm{in}} &\sim \left(q_1R_{1,\pm} +q_2R_{2,\pm} -1\right)e^{\iu \k_\pm \cdot x}\\
	&\sim\left(\frac{q_2+q_1}{2} \pm \frac{c_3}{L^2} \frac{q_2-q_1}{2}- 1+O(\omega)\right)e^{\iu \k_\pm \cdot x}.
	\end{align*}
	For all real $\omega$ with $0\leq \omega\leq K\sqrt{\delta}$ and for fixed $c_3\neq 0$, there exists a constant $A>0$ such that $|\omega-\omega_i| > A\delta$ for $i=1,2$. Then $\left( {C}^{0} + {\omega}C^{1,\alpha_0} -\frac{\omega^2|D_1|}{\delta v_b^2}I  \right)$ is invertible and its inverse satisfies $\left( {C}^{0} + {\omega}C^{1,\alpha_0} -\frac{\omega^2|D_1|}{\delta v_b^2}I  \right)^{-1} = O(\delta^{-1/2})$ uniformly in $\omega$. We have that
	$$d:=\det\left( {C}^{0} + {\omega}C^{1,\alpha_0} -\frac{\omega^2|D_1|}{\delta v_b^2}I  \right) = \left(-\iu  k_3 L^2- \frac{\omega^2|D_1|}{\delta v_b^2} \right)\left(2 C_{11}^0-\frac{\iu k_3 c_3^2}{L^2}- \frac{\omega^2|D_1|}{\delta v_b^2}\right) +O(\delta).$$
	Therefore
	\begin{multline*}
		\left( {C}^{0} + {\omega}C^{1,\alpha_0} -\frac{\omega^2|D_1|}{\delta v_b^2}I  \right)^{-1} \\ = \frac{1}{d}\left(\left(C_{11}^0-\frac{\iu k_3 c_3^2}{2L^2} \right)\begin{pmatrix} 1 & 1 \\ 1 & 1\end{pmatrix}  -\frac{\iu k_3 L^2}{2} \begin{pmatrix} 1 & -1 \\ -1 & 1\end{pmatrix} + \iu  (\alpha,0)\cdot\c_1\begin{pmatrix} 0 & -1 \\ 1 & 0\end{pmatrix} -\frac{\omega^2|D_1|}{\delta v_b^2}I\right).
		\end{multline*}
	Using \Cref{prop:sol}, we therefore have
	\begin{multline*}
		q = -\frac{1}{d} \left( \iu k_3 L^2\left(2C_{11}^0-\frac{\iu k_3 c_3^2}{L^2} -\frac{\omega^2|D_1|}{\delta v_b^2} \right)  + \iu  (\alpha,0)\cdot\c_1 (\iu k_3c_3)\right)\begin{pmatrix} 1 \\ 1\end{pmatrix} \\
		- \frac{1}{d} \left(\iu k_3c_3\left(-\iu k_3L^2-\frac{\omega^2|D_1|}{\delta v_b^2}  \right) - \iu  (\alpha,0)\cdot\c_1(\iu k_3L^2) \right)\begin{pmatrix} 1 \\ -1\end{pmatrix} + O(\delta^{1/2}),
	\end{multline*}
uniformly in $\omega$. Simplifying the above expression, we obtain
	$$
	q = \frac{-\iu k_3 L^2}{ -\iu  k_3 L^2- \frac{\omega^2|D_1|}{\delta v_b^2} } \left( 1 + \epsilon_1\right)\begin{pmatrix} 1 \\ 1\end{pmatrix} + \frac{-\iu k_3c_3}{ 2 C_{11}^0-\frac{\iu k_3 c_3^2}{L^2}- \frac{\omega^2|D_1|}{\delta v_b^2} } \left(1 - \epsilon_2\right)\begin{pmatrix} 1 \\ -1\end{pmatrix} + O(\delta^{1/2}),
	$$
	where 
	$$\epsilon_1 = \frac{\iu  (\alpha,0)\cdot\c_1 \frac{c_3}{L^2}}{\left(2 C_{11}^0-\frac{\iu k_3 c_3^2}{L^2}- \frac{\omega^2|D_1|}{\delta v_b^2}\right)}, \qquad \epsilon_2 = \frac{\iu (\alpha,0)\cdot\c_1\frac{L^2}{c_3}}{\left(-\iu  k_3 L^2- \frac{\omega^2|D_1|}{\delta v_b^2} \right)}.$$
	When $\omega$ is away from $\omega_1$, we have $\epsilon_2 = O(\delta^{1/2})$. For $\omega-\omega_1=O(\delta)$ we have 
	$$\frac{\iu k_3c_3}{ \left(2 C_{11}^0-\frac{\iu k_3 c_3^2}{L^2}- \frac{\omega^2|D_1|}{\delta v_b^2}\right)} \left(1 - \epsilon_2\right) = O(\delta^{1/2}).$$
	An analogous argument for $\omega$ close to $\omega_2$ shows that for $0\leq \omega \leq K\sqrt{\delta}$ it holds that
	$$
	q = \frac{-\iu k_3 L^2}{ \left(-\iu  k_3 L^2- \frac{\omega^2|D_1|}{\delta v_b^2} \right)}\begin{pmatrix} 1 \\ 1\end{pmatrix} + \frac{-\iu k_3c_3}{ \left(2 C_{11}^0-\frac{\iu k_3 c_3^2}{L^2}- \frac{\omega^2|D_1|}{\delta v_b^2}\right)}\begin{pmatrix} 1 \\ -1\end{pmatrix} +O(\delta^{1/2}).
	$$
	Then, as $x\to \pm\infty$, we can see that
	\begin{align*}
	u-u^{\mathrm{in}} &\sim\left(\frac{q_2+q_1}{2} \pm \frac{c_3}{L^2} \frac{q_2-q_1}{2}- 1+O(\omega)\right)e^{\iu \k_\pm \cdot x}\\
	&\sim\left(
	 \frac{-\iu k_3 L^2}{ -\iu  k_3 L^2- \frac{\omega^2|D_1|}{\delta v_b^2} }\pm \frac{\iu k_3c_3^2}{ L^2\left(2 C_{11}^0-\frac{\iu k_3 c_3^2}{L^2}- \frac{\omega^2|D_1|}{\delta v_b^2}\right)} - 1 + O(\delta^{1/2})\right)e^{\iu \k_\pm \cdot x},
	\end{align*}
	which gives
	$$t = \frac{-\iu k_3 L^2}{ -\iu  k_3 L^2- \frac{\omega^2|D_1|}{\delta v_b^2} } + \frac{\iu k_3c_3^2}{ L^2\left(2 C_{11}^0-\frac{\iu k_3 c_3^2}{L^2}- \frac{\omega^2|D_1|}{\delta v_b^2}\right)} + O(\delta^{1/2}), $$
	and
	$$r = \frac{-\iu k_3 L^2}{ -\iu  k_3 L^2- \frac{\omega^2|D_1|}{\delta v_b^2} } - \frac{\iu k_3c_3^2}{ L^2\left(2 C_{11}^0-\frac{\iu k_3 c_3^2}{L^2}- \frac{\omega^2|D_1|}{\delta v_b^2}\right)} - 1+O(\delta^{1/2}).
	$$
	In light of \Cref{thm:res0} we can rewrite this as
	$$S = \frac{\omega_1}{\omega_1-\omega}\begin{pmatrix} 1&1\\1&1\end{pmatrix} +\frac{2\iu\omega\Im(\omega_2)}{\omega_2^2-\omega^2}\begin{pmatrix} 1&-1\\-1&1\end{pmatrix} - \begin{pmatrix} 1&0\\0&1\end{pmatrix} + O(\delta^{1/2}),$$
	which proves the claim.
	\end{proof}

	\begin{remark}
		At the resonances, \textit{i.e.} when $\omega = 0$ or $\omega = \Re(\omega_2)$, the scattering matrix is given by
		$$S(0) = \begin{pmatrix} 0&1\\1&0\end{pmatrix} + O(\delta^{1/2}) \quad\text{and}\quad S(\Re(\omega_2)) = \begin{pmatrix} 0&-1\\-1&0\end{pmatrix} + O(\delta^{1/2}),$$
		corresponding to transmission peaks where the transmittance is close to $1$. The widths of these peaks are specified by the corresponding imaginary part $\Im(\omega_1)$ and $\Im(\omega_2)$.
	\end{remark}
	\begin{remark}
		\Cref{thm:phano} can be used to demonstrate the Fano-type transmission anomaly. 
		If we tune the parameters of the system so that $\Im(\omega_1)$ is large while $\Im(\omega_2)$ is small we can have, for small $\omega^*$, that
		$$
		\frac{\omega_1}{\omega_1-(\Re(\omega_2)-\omega^*)} \approx 
		\frac{\omega_1}{\omega_1-(\Re(\omega_2)+\omega^*)}  \approx 
		\frac{\omega_1}{\omega_1-\Re(\omega_2)}  =:t_1,$$
		where $t_1$ is not too small. In this case, the transmission coefficient is given by
		$$t(\Re(\omega_2) + \omega^*) \approx \frac{1}{1-\frac{\Re(\omega_2)}{\omega_1}} - \frac{1}{1-\frac{\omega^*}{\iu\Im(\omega_2)}}.$$
		In particular, at $\omega^* = \Re(\omega_2)\frac{\Im(\omega_2)}{\Im(\omega_1)}$, we can see that 
		$$t(\Re(\omega_2)+\omega^*) \approx 0, \quad t(\Re(\omega_2)-\omega^*) \approx 2t_1.$$
		We emphasize that $\omega^* > 0$ and that $t$ is close to zero at $\omega=\Re(\omega_2)+\omega^*$ and not at $\omega = \Re(\omega_2)-\omega^*$. In other words, we have an asymmetric transmission peak at $\omega=\Re(\omega_2)$. For some frequency slightly larger than $\Re(\omega_2)$ the transmittance will be close to zero, but for all frequencies slightly lower than $\Re(\omega_2)$ the transmittance will be nonzero.
	\end{remark}


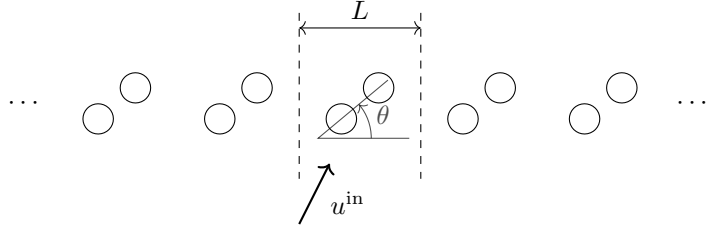
\begin{figure}[p]
	\centering
\begin{tikzpicture}[scale=0.4]
\begin{scope}[rotate=40]
\draw (0.8,0) circle (0.5);
\draw (-0.8,0) circle (0.5);
\end{scope}
\coordinate (oo) at (220:1.8);
\coordinate (a) at ($(0:3)+(oo)$);
\coordinate (b) at ($(40:3)+(oo)$);
\coordinate (o) at ($(0,0)+(oo)$);
\draw[opacity=0.7] (a) -- (o) -- (b);
\pic [draw, opacity=0.7, ->, "$\theta$", angle eccentricity=1.3,angle radius=0.7cm] {angle = a--o--b};
\node at (11,0){$\cdots$};
\node at (-11,0){$\cdots$};
\draw[->,thick] (-2,-4) -- (-1,-2) node[pos=0.7,below right]{$u^{\mathrm{in}}$};
\draw[dashed] (-2,3) -- (-2,-2.5) (2,3) -- (2,-2.5);
\draw[<->] (-2,2.5) -- (2,2.5) node[pos=0.5,above]{$L$};
\begin{scope}[xshift=4cm,rotate=40]
\draw (0.8,0) circle (0.5);
\draw (-0.8,0) circle (0.5);
\end{scope}
\begin{scope}[xshift=-4cm,rotate=40]
\draw (0.8,0) circle (0.5);
\draw (-0.8,0) circle (0.5);
\end{scope}
\begin{scope}[xshift=8cm,rotate=40]
\draw (0.8,0) circle (0.5);
\draw (-0.8,0) circle (0.5);
\end{scope}
\begin{scope}[xshift=-8cm,rotate=40]
\draw (0.8,0) circle (0.5);
\draw (-0.8,0) circle (0.5);
\end{scope}
\end{tikzpicture}
\caption{The front view of a symmetric metascreen with an incident plane wave $u^{\mathrm{in}}$. In this case, we have circular resonators arranged in a $\mathcal{P}$-symmetric dimer that is inclined at an angle of $\theta$ to the plane of the metascreen.} \label{fig:2D}
\end{figure}

\begin{figure}[p] 
	\begin{subfigure}[b]{0.45\linewidth}
		\begin{center}
			\includegraphics[width=1\linewidth]{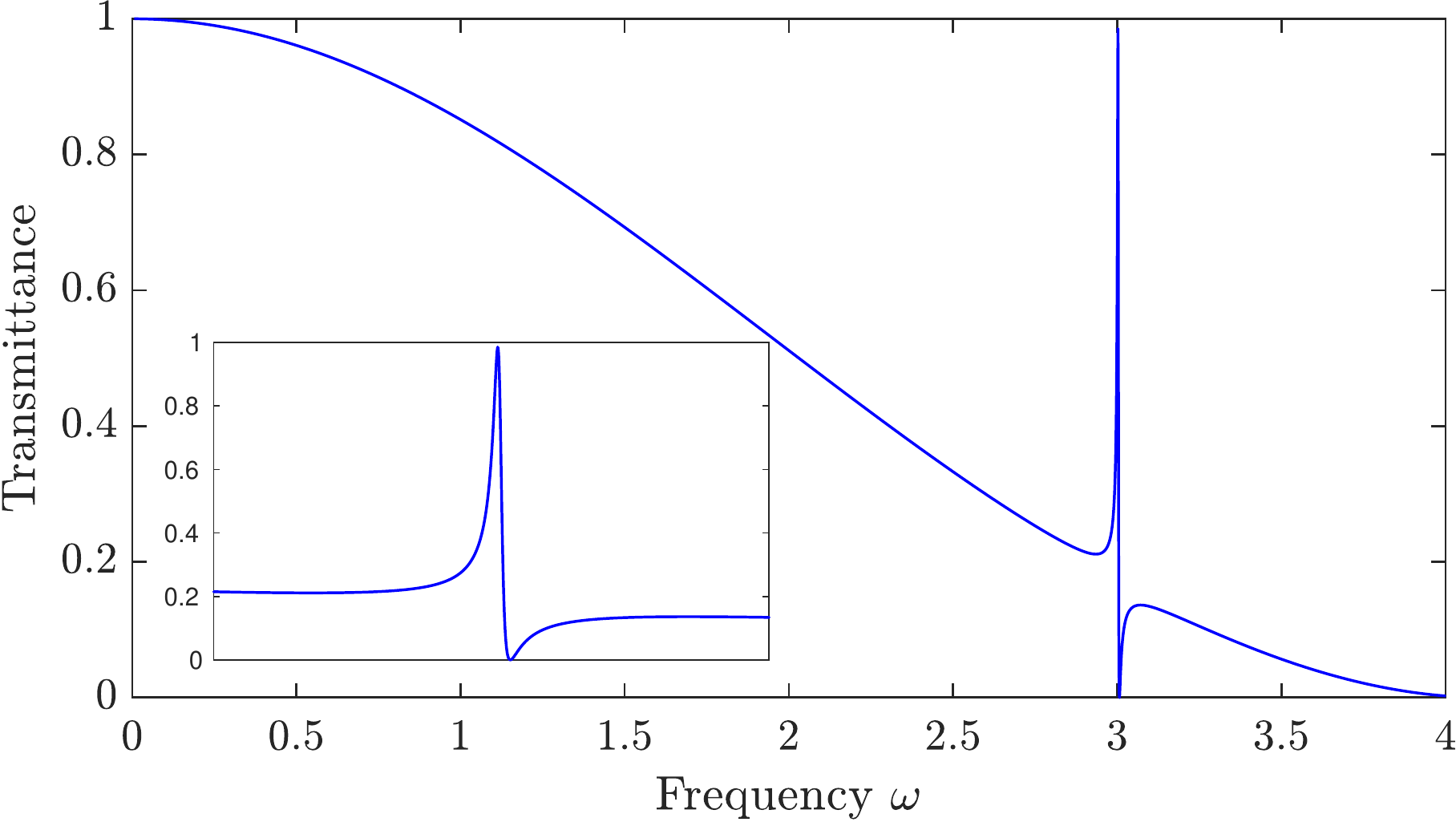}
		\end{center}
		\caption{Transmission spectrum.	} \label{fig:fano0T}
	\end{subfigure}
	\hspace{10pt}
	\begin{subfigure}[b]{0.45\linewidth}
		\begin{center}
			\includegraphics[width=1\linewidth]{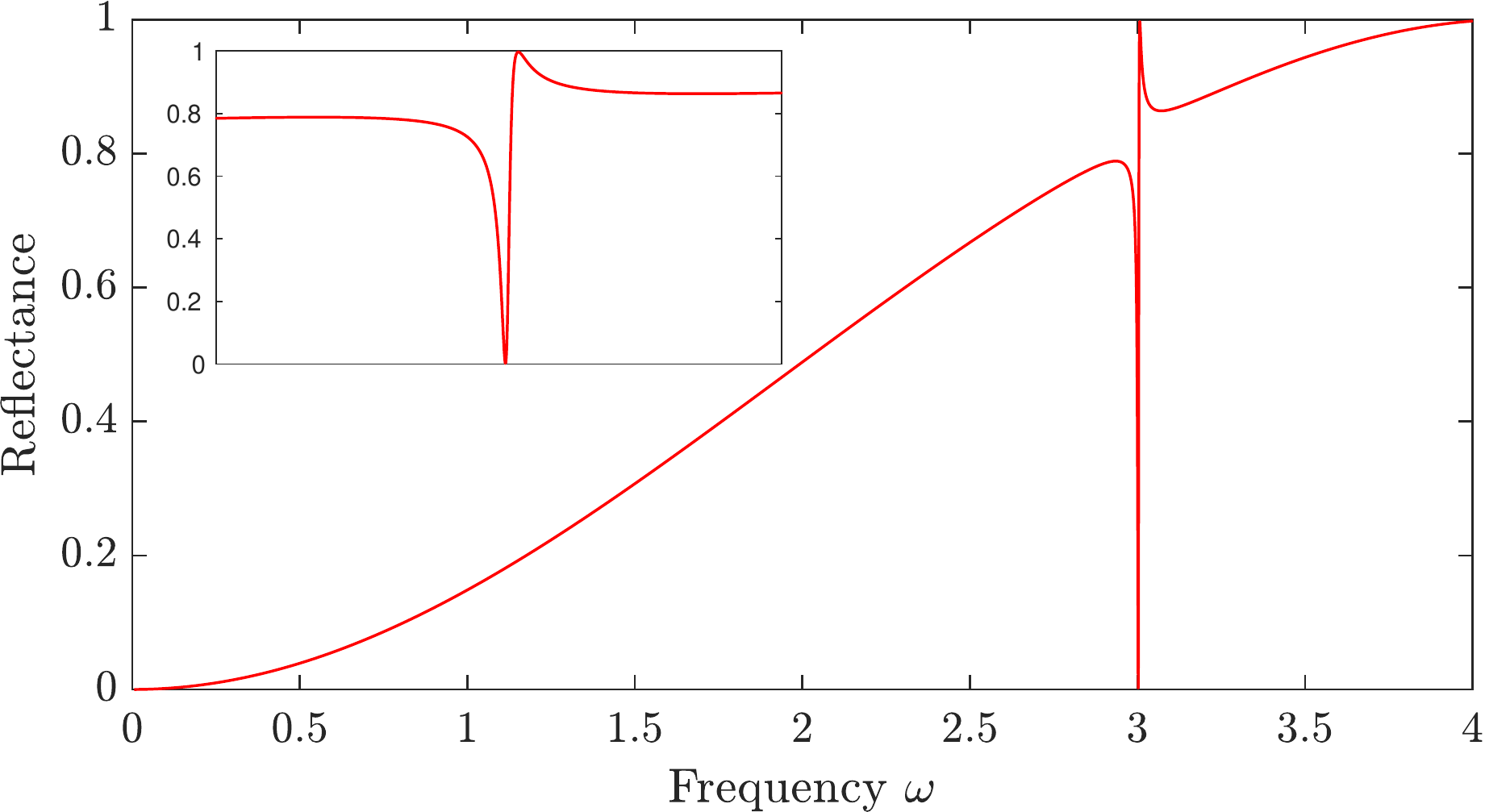}
		\end{center}
		\caption{Reflection spectrum.}\label{fig:fano0R}
	\end{subfigure}
	\caption{Transmittance (a) and reflectance (b) in the case $\delta =  0.02$, computed using the multipole discretization method. The asymmetric Fano-type transmission curve is clearly visible. Choosing $\delta$ rather large makes the Fano-type asymmetry more pronounced. Here, we use $\theta = 0.025\pi$. }\label{fig:fano}
\end{figure}

\begin{figure}[p] 
	\begin{subfigure}[b]{0.45\linewidth}
		\begin{center}
			\includegraphics[width=1\linewidth]{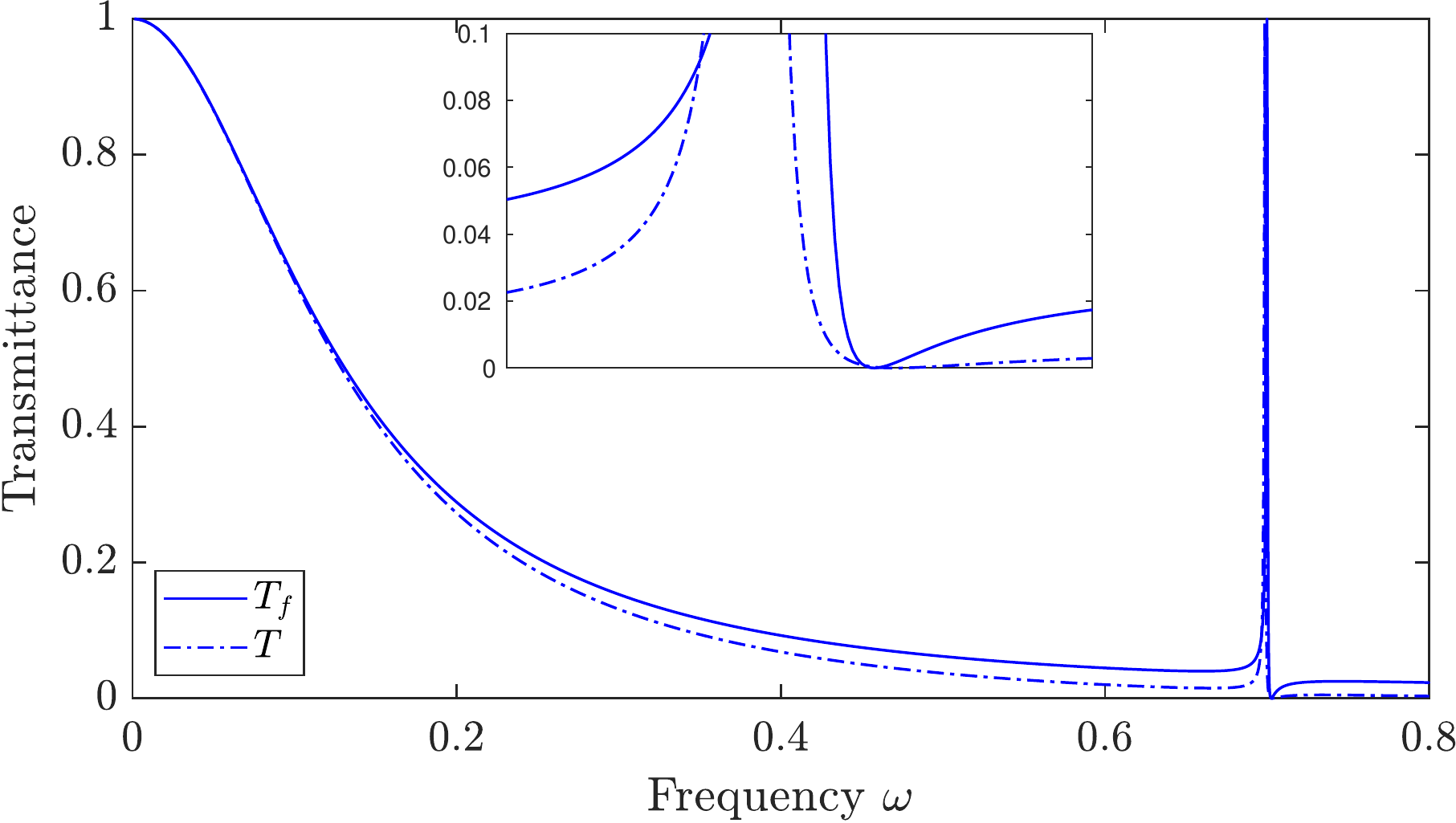}
		\end{center}
		\caption{Transmission spectrum.} \label{fig:fano2}
	\end{subfigure}
	\hspace{10pt}
	\begin{subfigure}[b]{0.45\linewidth}
		\begin{center}
			\includegraphics[width=1\linewidth]{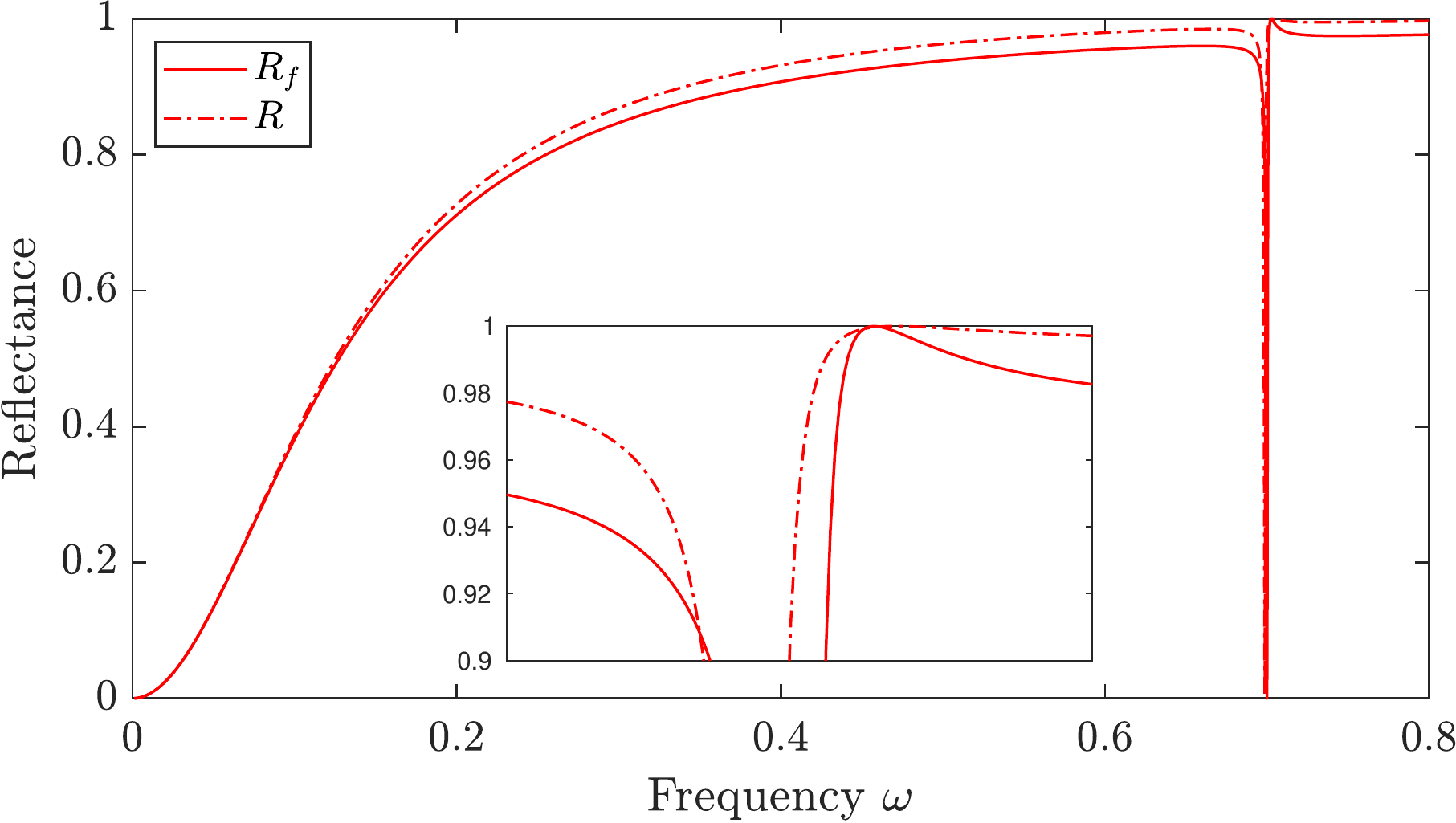}
		\end{center}
		\caption{Reflection spectrum.}\label{fig:fano2f}
	\end{subfigure}
	\caption{Transmittance (a) and reflectance (b) in the case $\delta = 10^{-3}$, computed using the multipole discretization (dashed line) and the asymptotic formulas (solid line). Compared to \Cref{fig:fano}, the smaller value of $\delta$ here means that the resonances occur in the subwavelength regime and that the asymptotic formulas provide a good approximation. Here, we use $\theta = 0.05\pi$.}\label{fig:fanos2}
\end{figure}

\begin{figure}
	\begin{subfigure}[b]{0.45\linewidth}
		\begin{center}
			\includegraphics[width=1\linewidth]{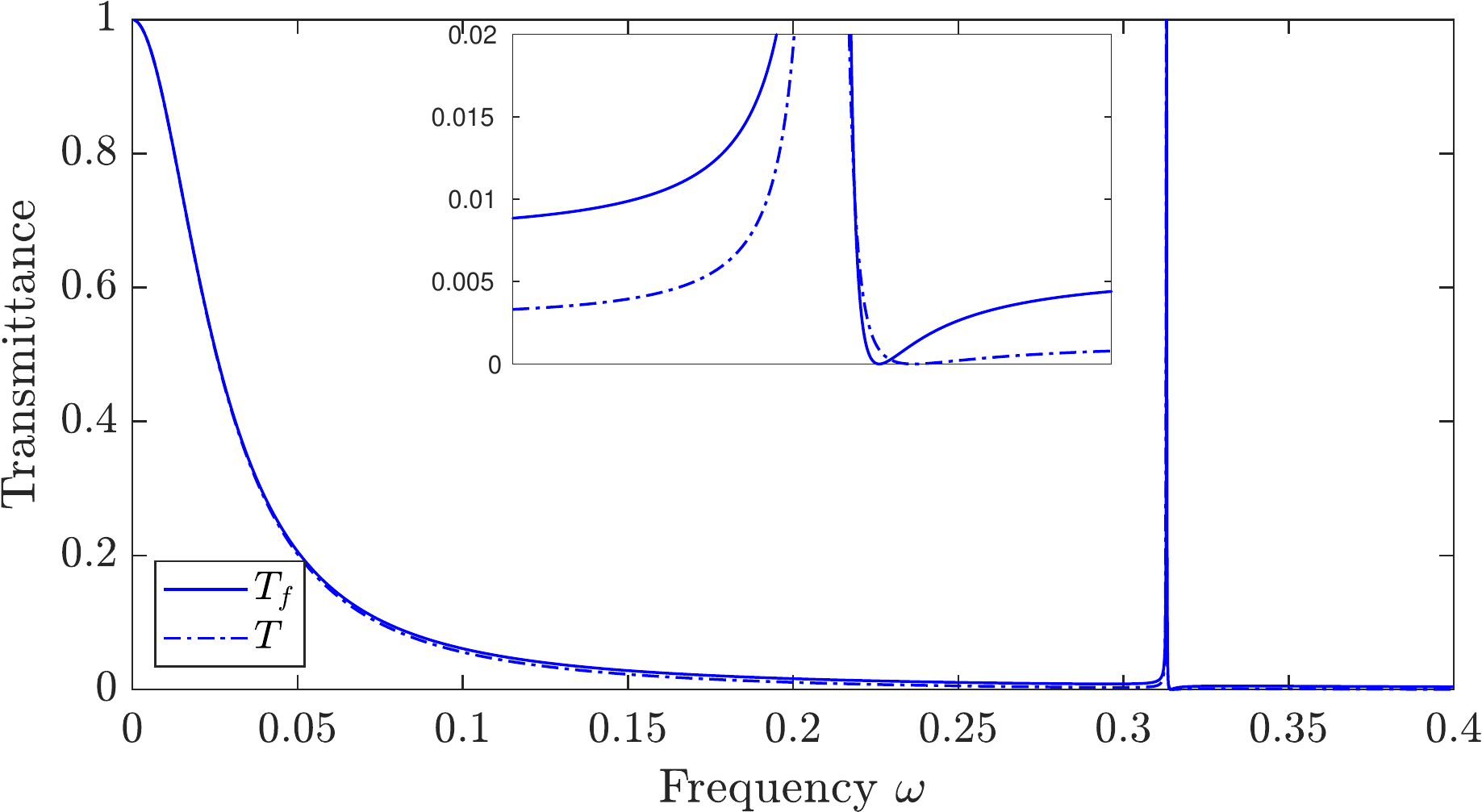}
		\end{center}
		\caption{Transmission spectrum.} \label{fig:fano3}
	\end{subfigure}
	\hspace{10pt}
	\begin{subfigure}[b]{0.45\linewidth}
		\begin{center}
			\includegraphics[width=1\linewidth]{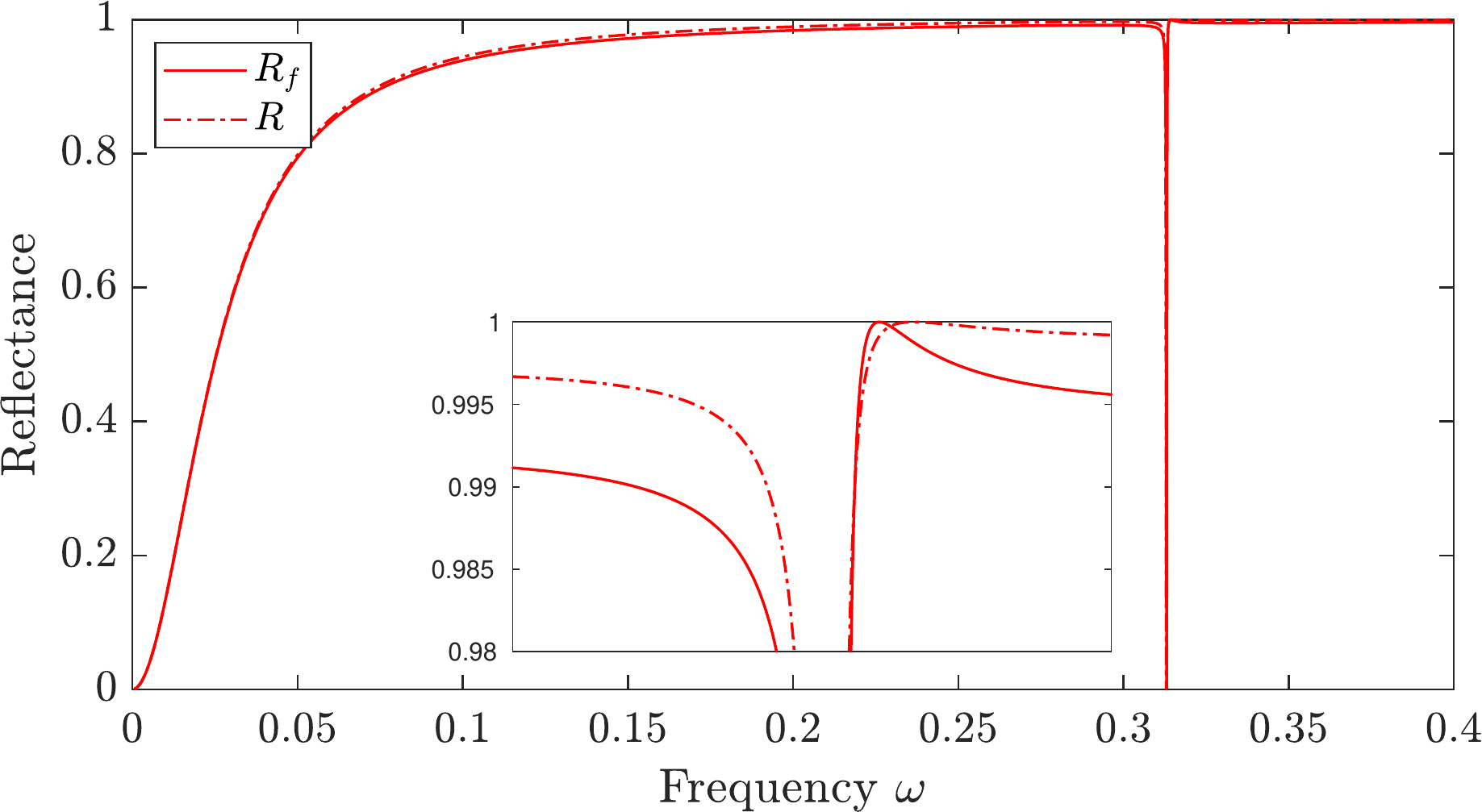}
		\end{center}
		\caption{Reflection spectrum.}\label{fig:fano3f}
	\end{subfigure}
	\caption{Transmittance (a) and reflectance (b) in the case $\delta = 2\cdot 10^{-4}$, computed using the multipole discretization (dashed line) and the asymptotic formulas (solid line). Here, we use $\theta = 0.05\pi$.} \label{fig:fanos3}
\end{figure}

\begin{figure}
	\begin{center}
		\includegraphics[width=0.6\linewidth]{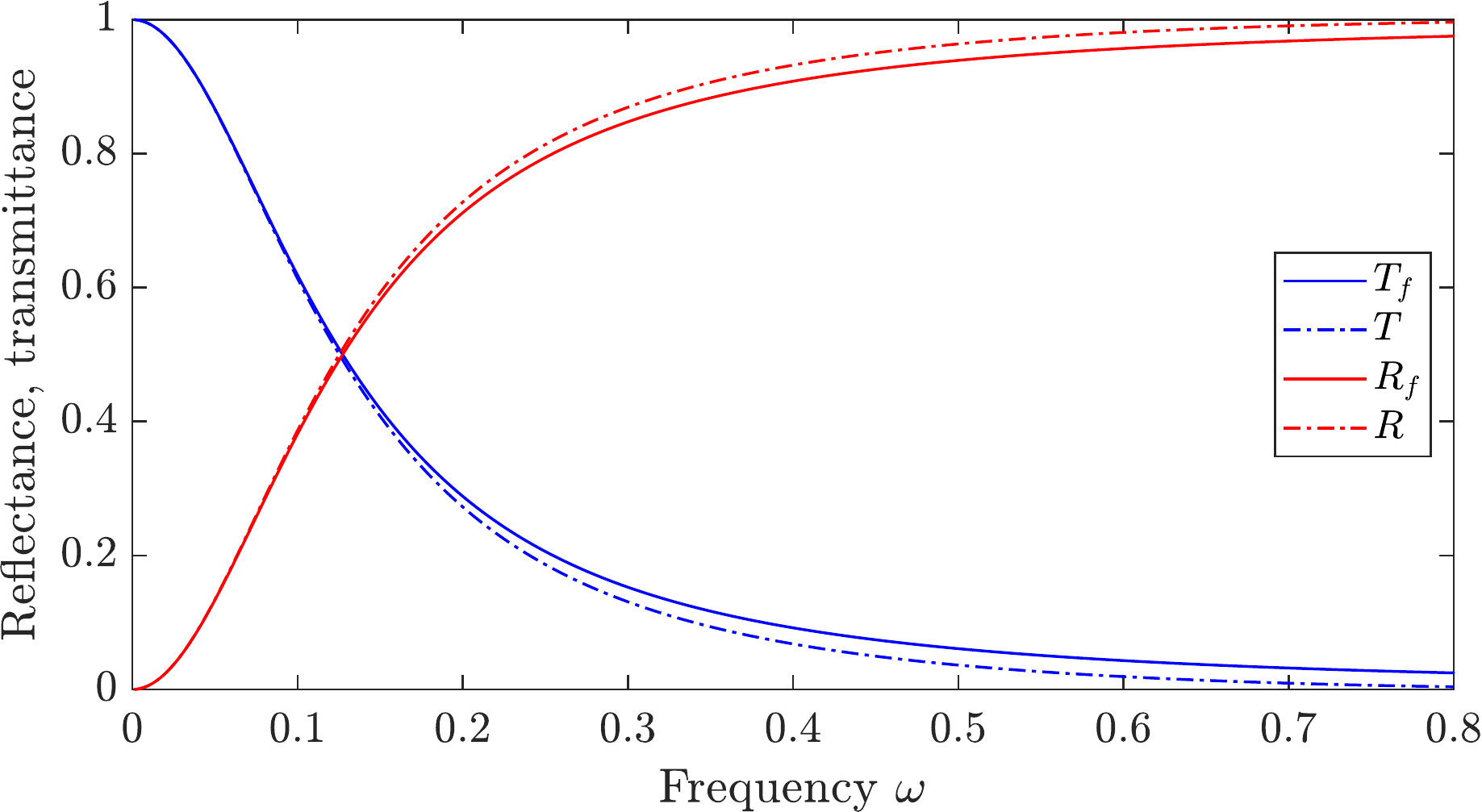}
	\end{center}
	\caption{In a symmetric structure, there is a bound state in the continuum. Here, we repeat the analysis from \Cref{fig:fanos2} (with $\delta = 1/1000$) but with $\theta = 0$. We observe that the transmission peak around $\omega\approx 0.7$ is not present in this case, consistent with \Cref{prop:trans0} and the fact that corresponding resonant mode is a bound state in the continuum. This was computed using both the multipole discretization (dashed line) and the asymptotic formulas (solid line).} \label{fig:fanoBIC}
\end{figure}	

\subsection{Numerical illustrations}\label{sec:num}
Here, we compute numerically the transmission and reflection spectra of a two-dimensional analogue of the structure studied above. We assume that $D$ consists of two disks $D_1$ and $D_2$, which are both of radius $R_D$ and are separated by a distance $d$. Moreover, we denote by $\theta$ the angle at which $D$ is rotated from being parallel to the $x_1$-axis, as depicted in \Cref{fig:2D}. In order to achieve a Fano-type resonance, we will choose a small but nonzero $\theta$.

We use the material parameters $L=1$, $R_D=0.05$, $d=0.3$, $\alpha_0 = 0$ and $v=v_b=1$ (these are the same parameters used to derive the band structure in \Cref{sec:band}). We compute the transmittance $T = |t|^2$ and reflectance $R=|r|^2$ in two different ways: using the asymptotic formulas that were derived in \Cref{sec:Sfano} (denoted by $T_f$ and $R_f$) and by discretizing the operator $\A^\omega$ using the multipole method (denoted by $T$ and $R$). For details on the multipole discretization method, we refer to, \textit{e.g.} \cite{ammari2020exceptional, ammari2017subwavelength}.

Figures \ref{fig:fano}--\ref{fig:fanos3} show numerically computed transmission spectra for different values of $\delta$. In \Cref{fig:fano} we chose $\theta= 0.025\pi$ while in Figures~\ref{fig:fanos2} and \ref{fig:fanos3} we chose $\theta= 0.05\pi$. All cases demonstrate Fano-type transmission anomalies, where the transmittance is close to zero for frequencies slightly above the second peak. For larger values of $\delta$, the widths of the first, broad transmission peaks are larger and the Fano-type transmission more pronounced. As expected, the asymptotic formulas are more accurate for smaller values of $\delta$.

In \Cref{fig:fanoBIC} we show the transmission spectra for the same parameters as \Cref{fig:fanos2} but with $\theta = 0$, meaning that the structure is $\P_3$-symmetric. We observe that there is no the sharp transmission peak in this case, which is expected since the eigenmodes are bound states in the continuum and are not excited by incident waves from the far field.
	
\section{Concluding remarks}
In this work, we have studied the existence of Fano-type resonances for high-contrast resonators in the subwavelength regime. We have explicitly characterized the Fano-type resonance in terms of the periodic capacitance matrix. In the symmetric case, we have proved that the structure supports bound states in the continuum. When the symmetry is perturbed, we showed that the states will interact with the far field and give rise to an asymmetric Fano-type transmission line shape. 

Building on this work, there have also been various attempts to use the topological properties of periodic structures to produce Fano-type responses that are robust as a result of being \emph{topologically protected} \cite{gao2018fano,ni2018fano, wang2020robust, zangeneh2019topological}. With the analysis of \cite{ammari2020topologically} in mind, we expect that it will be possible to create robust Fano-type resonances in the subwavelength regime by coupling  a ``continuum'' to a ``discrete'' state which is topologically protected. Moreover, in a related work \cite{ammari2021functional} we will study a large system of finitely many resonators and examine the extent to which its Fano-type transmission and reflection behaviours can be approximated by the infinite system that was studied here.

\bibliographystyle{abbrv}
\bibliography{boundstate}{}

\end{document}